\documentclass[12pt, letterpaper]{amsart}

\usepackage[margin=1in]{geometry}

\usepackage{amsmath, amssymb, amsthm, amsfonts, amsxtra}
\usepackage{graphicx}
\usepackage{listings}
\usepackage{xcolor}
\usepackage{url}
\usepackage{comment,enumitem}

\newcommand{\indentalign}{\hspace{0.3in}&\hspace{-0.3in}}

\newcommand{\defeq}{\stackrel{\rm{def}}{=}}

\newcommand{\bds}[1]{\boldsymbol{#1}}
\newcommand{\ext}{\textnormal{ext}}
\newcommand{\tot}{\textnormal{tot}}
\renewcommand{\mod}{\;\mathrm{mod}\;}

\newcommand{\R}{\mathbb{R}}
\newcommand{\T}{\mathbb{T}}
\newcommand{\Om}{\Omega}
\newcommand{\vperp}{v_{\!\perp}}
\newcommand{\xperp}{x_{\!\perp}}

\newcommand{\dd}{\,\mathrm{d}}
\newcommand{\curl}{\operatorname{curl}}
\newcommand{\diver}{\operatorname{div}}

\newtheorem{theorem}{Theorem}[section]

\newtheorem{proposition}[theorem]{Proposition}
\newtheorem{lemma}[theorem]{Lemma}

\theoremstyle{remark}
\newtheorem{remark}[theorem]{Remark}
\newtheorem{example}[theorem]{Example}

\numberwithin{equation}{section}
\numberwithin{figure}{section}
\numberwithin{table}{section}

\title[Equilibria for Vlasov-Maxwell]{Equilibria for the Vlasov-Maxwell system related to plasma confinement}
\author{Justin Holmer}
\address{Brown University}
\email{justin\_holmer@brown.edu}

\author{Katherine Zhiyuan Zhang}
\address{Northeastern University}
\email{zhi.zhang@northeastern.edu}

\begin{document}

\maketitle

\begin{abstract}
By working in a setting of azimuthal and $z$-directional invariance, we reduce the existence of electrically neutral, magnetically driven, equilibria of the two-species Vlasov-Maxwell (VM) equation to a second-order ODEs in the radial variable.  In contrast to previous works (Degond \cite{Degond1990}, Batt \& Fabian \cite{BattFabian1993}, Weber \cite{Weber_2020b}), we consider specific situations that correspond physically to $z$-pinch, $\theta$-pinch, and screw-pinch equilibria in plasma physics.  Semi-explicit solutions in these situations are provided.
\end{abstract}

\tableofcontents

\section{Introduction}

Our purpose is to investigate the existence of semi-explicit physically relevant equilibria for the two-species Vlasov-Maxwell (VM) system \eqref{E:Vlasov}, \eqref{E:Maxwell}, \eqref{E:charge_current}, in the cylindrically and vertically symmetric setting.  The densities are denoted $f^\pm$, the electric field $\bds E$, and magnetic field $\bds B$, and we make use of the electric potential $\phi$ and magnetic potential $\bds A$ satisfying the Coulomb gauge (see \eqref{E:potentials}).  By \emph{cylindrically and vertically symmetric}\footnote{Other adjectives are more common:  what we refer to as cylindrical symmetry ($\theta$ translational invariance) is often called axial symmetry and could also be called azimuthal symmetry, and what we call vertical symmetry $(z$ translational invariance) could be called axial invariance.  We have avoided these terms since the dual use of axial could lead to confusion.  The word \emph{axial} by itself refers to the $z$-direction, so \emph{axial symmetry} could refer to symmetry \emph{around} the axis, \emph{in the direction of} the axis, or \emph{across} (reflection in) the axis.}   we mean that we take $x,y\in \mathbb{R}$ but $z\in \mathbb{T}$, implement cylindrical coordinates, and assume that
\begin{enumerate}[left=0pt,label=$\bullet$]
\item the densities $f^\pm$ in the cylindrical coordinates are functions of $(r,v_r,v_\theta,v_z)$, have no explicit dependence on $\theta$, but depend on $\theta$ implicitly through $v_r$ and $v_\theta$.
\item the scalar field $\phi$ is independent of $z$ and $\theta$
\item the vector fields $\bds A$, $\bds E$, $\bds B$, when written in the cylindrical frame, have \emph{components} independent of $\theta$ and $z$.  For example, when the magnetic field is decomposed as $\bds B = B_r \bds e_r + B_\theta \bds e_\theta + B_z \bds e_z$, then the basis coefficients $B_r$, $B_\theta$, $B_z$ are independent of $\theta$ and $z$.
\end{enumerate}
(Notation for cylindrical coordinates is reviewed in \S\ref{S:cyl}.) In this setting, the VM system simplifies to \eqref{E:ax_vert_Maxwell}, \eqref{E:ax_vert_Vlasov}.  Field components of equilibria satisfy \eqref{E:M-cyl} and the densities are naturally expressed in terms of the three particle trajectory invariants: the energy $e^\pm$, the $z$-component of angular momentum $\ell_z^\pm$, and the $z$-component of linear momentum $p_z^\pm$.
(Particle trajectories and their invariants are reviewed in \S\ref{S:equil}.)  In particular, we are interested in electrically neutral equilibria ($\bds E=0$, $\phi=0$, $\rho=0$) with strong magnetic confining that fall into the following three categories
\begin{enumerate}[left=0pt,label=$\bullet$]
\item $z$-pinch type, in which $f^\pm$ depends only on $e^\pm$, $p_z^\pm$, the magnetic field points in the $\theta$-direction, and the current points in the $z$-direction.
\item $\theta$-pinch type, in which $f^\pm$ depends only on $e^\pm$, $\ell_z^\pm$, the magnetic field points in the $z$-direction, and the current points in the $\theta$-direction.
\item screw pinch, where the current has both $\theta$- and $z$-components, so that current paths are helical.
\end{enumerate}

These equilibria are relevant in plasma physics as canonical magnetic confinement configurations, originally studied in the context of laboratory fusion devices and
current-carrying plasmas, and they continue to serve as basic model geometries for
understanding the interaction between particle motion and self-generated magnetic
fields.

While these configurations are often discussed within the framework of ideal or resistive magnetohydrodynamics, a kinetic description is preferable in regimes where collisionless effects, velocity-space structure, or species-dependent dynamics are important.  The Vlasov-Maxwell system provides a first-principles model in which
equilibria are determined self-consistently from particle distributions and
electromagnetic fields, without recourse to fluid closure assumptions.  This makes the
VM setting particularly well-suited for rigorous stability analysis, especially for
equilibria that depend on multiple invariants of particle motion.

In this work, we construct semi-explicit electrically neutral
Vlasov-Maxwell equilibria of $z$-pinch, $\theta$-pinch, and screw-pinch type within
the cylindrically and vertically symmetric geometry described above.  By \emph{semi-explicit}, we mean that we present explicit forms for the Vlasov densities $f^\pm = \mu^\pm(e^{\pm}, \ell_z^\pm, p_z^\pm)$, and then by grouping constants and scaling, we obtain streamlined second-order ODEs for the magnetic potential field functions $A_z$ and $A_\theta$.  While these ODEs usually do not have explicit solutions in the sense of explicit formulae, their properties and behavior can be qualitatively analyzed in detail and numerical calculations are straightforward.  Our main motivation is, in subsequent work, to analyze the linear stability and instability of these equilibria using the framework developed by Lin \& Strauss \cite{LinStraussLinear2007, LinStrauss2008}, and further explained schematically by Lin \cite{Lin2022} using the separable Hamiltonian language of Lin \& Zeng \cite{LinZeng2022}. 

The geometry of the domain is $r>0$, $\theta \in [0,2\pi)$ and $0\leq z < \zeta$ with periodic boundary conditions connecting $z=0$ and $z=\zeta$.  Recall that we are considering a two-species solution, one with charge $q^-<0$ and one with charge $q^+>0$.  We take the following given physical constants.
\begin{enumerate}[left=0pt,label=$\bullet$]
 \item $\beta>0$ is a radial length scale, with SI units of m,
\item $\zeta>0$ is the $z$-length of the apparatus, with SI units of m.
\item  $m^\pm$ are the masses of the two species, with SI units of kg.
\item $q^\pm$ are the charges of the species, with SI units of C, with $q^-<0$ and $q^+>0$.
\item  $\sigma^\pm>0$ are thermal speeds of the two species, with units of $\text{m}/\text{s}$, so that $\sim m^\pm (\sigma^\pm)^2$ are the thermal temperatures, satisfying
\begin{equation}
\label{E:lambda-def}
\lambda \defeq \frac{q^+}{m^+\sigma^+} = \frac{-q^-}{m^-\sigma^-}
\end{equation}
\item The constant $\mu_0=1.26 \cdot 10^{-6} \text{ m kg/C}^2$ is the permeability of free space.
\end{enumerate}
The equilibria below are described though an additional constant $a>0$, with SI units of $\text{C}$, that can be freely selected, that ultimately affects the total number of particles.  With these physical parameters specified, we introduce the dimensionless constant
\begin{equation}\label{E:d-def}
d \defeq \frac{ (2\pi)^{3/2} e^{1/2} \, a \, \mu_0}{  \zeta} \left( \sum_\pm \frac{|q^\pm|}{m^\pm} \right) >0
\end{equation}
The particle trajectory invariants are given by
$$e^\pm = \tfrac12 m^\pm (v_r^2+v_\theta^2+v_z^2) + q^\pm \phi(r) \,, \quad
p_z^\pm = m^\pm v_z + q^\pm A_z(r) \,, \quad 
\ell_z^\pm = r(m^\pm v_\theta + q^\pm A_\theta(r))$$
which are, respectively, the energy, $z$-component of linear momentum, and $z$-component of angular momentum (alternatively, the $\theta$-\emph{conjugate} or \emph{canonical} momentum).  Note that in each theorem below, the field solution is fixed first as a solution to a reference dimensionless  second-order ODE ($A_z(r)$ in the case of Theorems \ref{T:z-pinch-const}, \ref{T:z-pinch-var}, $A_\theta(r)$ in the case of Theorem \ref{T:theta-pinch-var}, and a hybrid of $A_\theta(r)$ and $A_z(r)$ in the case of Theorem \ref{T:screw-pinch-var}) and \emph{then} the densities $f^\pm$ are expressed in terms of $e^\pm$, $p_z^\pm$, and $\ell_z^\pm$ which themselves depend upon the fields $A_z(r)$ and $A_\theta(r)$ ($\phi(r)=0$ in all cases we treat). 

\begin{theorem}[$z$-pinch, constant drift equilibria]
\label{T:z-pinch-const}
Select $a>0$ (in units of charge) and then let $d$ be the dimensionless constant \eqref{E:d-def}.  Let $\bar A_z(\bar r)$ solve the dimensionless second-order ODE
\begin{equation}
\label{E:z-const-ODE}
(-\partial_{\bar r}^2 - \frac{1}{\bar r} \partial_{\bar r}) \bar A_z = d \, e^{\bar A_z}
\end{equation}
with boundary conditions 
$$\bar A(0) \text{ finite}  \quad \text{and} \quad \partial_{\bar r} \bar A(0) =0 \,.$$
Then in terms of $\bar A_z(\bar r)$, we can produce an equilibrium solution to the VM system as follows. 
Let $r=\beta \bar r$, take the electric potential $\phi=0$, the cylindrical components of the magnetic potential $\bds A$ to be 
$$A_r=0\,, \quad A_\theta=0\,, \quad A_z(r) = \lambda^{-1}\bar A_z(\bar r) \,,$$
and the densities be $f^\pm (r, v_r,v_\theta,v_z) = \mu^\pm( e^\pm(v_r,v_\theta,v_z),p_z^\pm(r, v_z))$, where
\begin{equation}
\mu^\pm(e^\pm, p_z^\pm) = \frac{a}{\beta^2\zeta (\sigma^\pm)^3|q^\pm|} \exp \left( \frac{-  e^\pm}{m^\pm (\sigma^\pm)^2} \pm \frac{  p_z^\pm}{m^\pm \sigma^\pm} \right)
\end{equation}
Then $(f^\pm, \phi, \bds A)$ is an equilibrium solution to the VM system.   To give the explicit $r$ and $\bds v$ dependence of the densities, we complete the square in the exponent:
\begin{equation}\label{E:z-mu-pm-2}
\mu^\pm(r,  v_r,  v_\theta, v_z)   =  \frac{n^\pm(\bar r)}{(2\pi)^{3/2} (\sigma^\pm)^3}  \,\exp \Big[ -\frac12 \Big( \frac{v_r}{\sigma^\pm}\Big)^2 - \frac12\Big( \frac{ v_\theta}{\sigma^\pm}\Big)^2 - \frac12\big( \frac{ v_z}{\sigma^\pm} \pm 1\Big)^2\Big]
\end{equation}
where the number densities for two species are
\begin{equation}\label{E:z-mu-pm-3}
n^\pm(\bar r) = \frac{(2\pi)^{3/2}}{\beta^2\zeta} \frac{a}{|q^\pm|} \exp[  \bar A_z(\bar r) ]
\end{equation}
The resulting charge density $\rho=0$ and the current density $\bds j$ has cylindrical components 
\begin{equation}
\label{E:z-const-curr}
j_r=0\,, \quad j_\theta=0\,, \quad j_z(r) = \frac{d}{\mu_0 \beta^2 \lambda}  \exp[\bar A_z(\bar r)]
\end{equation}
The electric field $\bds E=0$ and the magnetic field $\bds B$ has cylindrical components
$$ B_r=0\,, \quad  B_z=0\,, \quad B_\theta(r) = \frac{\sqrt 2}{\lambda \beta} \bar B_\theta(\bar r)\,, \quad \text{where }\bar B_\theta = -\partial_{\bar r} \bar A_z $$
\end{theorem}

This is proved in \S\ref{SS:z-const}.  The $d$ coefficient on the right side of \eqref{E:z-const-ODE} can be removed by replacing $\bar r$ by $\sqrt{d}\, \bar r$.  Therefore, it suffices to consider \eqref{E:z-const-ODE} with any convenient fixed choice of $d$, and the most convenient choice is $d=8$ (see \eqref{E:Bennett} below). 

As $\bar r\to +\infty$, there are no solutions to \eqref{E:z-const-ODE} for which $\bar A(\bar r)$ has a finite limit. Physically we would like the current density to vanish as  $\bar r\to \infty$, so \eqref{E:z-const-curr} requires that 
$$\bar A_z(\bar r) \to -\infty \quad \text{as} \quad \bar r \to +\infty$$
From \eqref{E:z-const-curr}, $\bar A_z(\bar r)$ must approach, as $\bar r \to +\infty$, a solution to $(-\partial_{\bar r}^2 - \frac{1}{\bar r} \partial_{\bar r}) \bar A_z = 0$, so
$$\bar A_z(\bar r) \sim -\log \bar r \quad \text{as} \quad \bar r \to +\infty$$
It turns out that \eqref{E:z-const-ODE} with $d=8$ and these boundary conditions has an explicit solution:  
\begin{equation}\label{E:Bennett}
\bar A_z(\bar r) = -2 \log \Big( 1+ e^D\bar r^2\Big) + D \,, \quad -\infty<D<+\infty
\end{equation}
with $\bar A_z(0)=D$.    This is called the \emph{Bennett solution} \cite{Bennett1934}.    At the end of \S\ref{SS:z-const} we show that \eqref{E:z-const-ODE} can be transformed to an equation in Hamiltonian form, which allows for classification of all solutions.  However, the solutions with physical boundary conditions indicated above are given by \eqref{E:Bennett}.  The corresponding scaled magnetic field is (see Figure \ref{F:z-pinch-const})
\begin{equation}\label{E:Bphi}
\bar B_\theta(\bar r) = \frac{4\bar r}{(e^{-D}+\bar r^2)}
\end{equation}
This decays as $\sim 1/\bar r$ as $\bar r \to +\infty$, which means that $\bds B$ just fails to belong to $L^2(\mathbb{R}^3)$.  In fact, \cite[Theorem 2, p. 275]{GlasseyStrauss1984} implies that there are no nontrivial finite-energy equilibrium solutions to VM in $\mathbb{R}^3$.  That theorem defines the local energy of a VM solution to be
$$E_R(t) = \int_{|\bds x|<R} \left( |\bds E(\bds x, t)|^2 +  |\bds B(\bds x, t)|^2 +  \int_{\bds v} |\bds v|^2 f(\bds x, \bds v, \bds t) \, d \bds v \right) \, d\bds x$$
and shows that, for a finite energy solution, there exists a sequence $t_n\nearrow +\infty$ such that $E_R(t_n) \to 0$.  However, for an equilibrium solution, $E_R(t)$ is constant in $t$.  

Upon substituting \eqref{E:Bennett} with $D=0$ into \eqref{E:z-mu-pm-3},
$$n^\pm(\bar r) = \frac{1}{\beta^2\zeta} \frac{a}{|q^\pm|} \frac{1}{(1+ \bar r^2)^2}$$
indicating a spatial decay rate for the density of $\sim r^{-4}$.  From \eqref{E:z-const-curr}, the $z$-current density is given by the same expression.  While the charges balance, there is a net current in the $z$-direction, and hence the name $z$-pinch for such solutions.

The expression for $\mu^\pm$ above shows a \emph{constant drift} in the $z$-velocity of the two species equal to $\pm\sigma^\pm$, which are of opposite signs.  The next family of solutions that we construct have variable drift in the $z$-velocity and have $z$-current density decaying faster than $\sim r^{-4}$.

\begin{figure}
\includegraphics[scale=0.65]{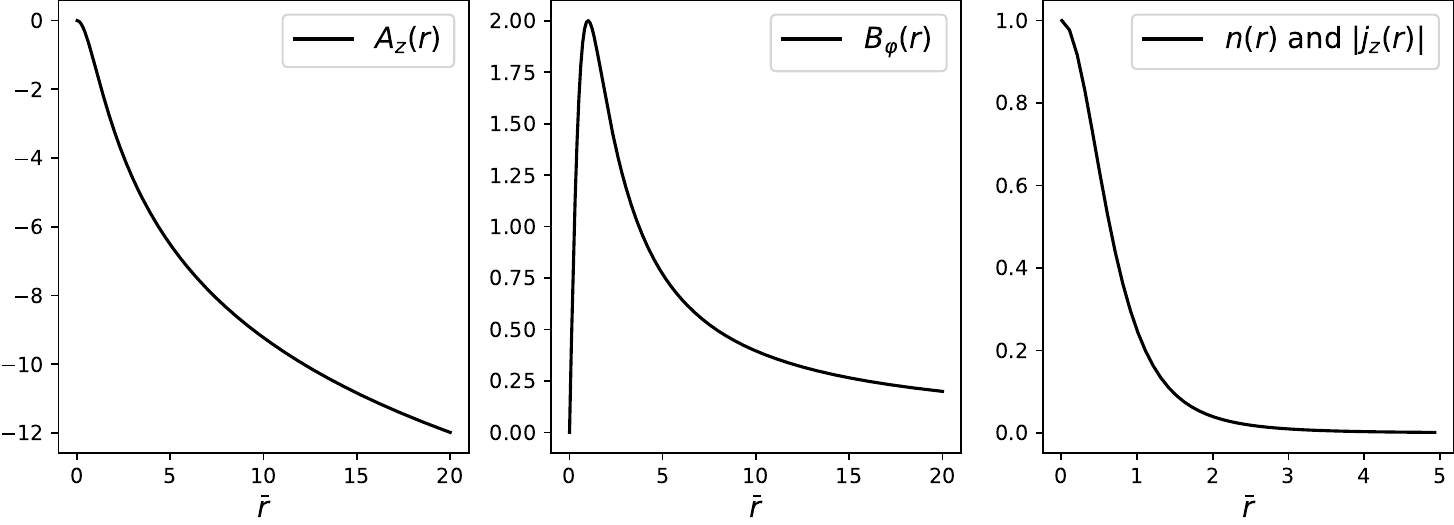}
\caption{(left) plot of $\bar A_z(\bar r)$ solving \eqref{E:z-const-ODE} in Theorem \ref{T:z-pinch-const} given explicitly by \eqref{E:Bennett} with $D=0$, which diverges logarithmically as $\bar r \to +\infty$; (middle) plot of the corresponding magnetic field $\bar B_\theta(\bar r)=-\partial_{\bar r}\bar A_z(\bar r)$ given explicitly by \eqref{E:Bphi}, which decays like $B_\theta(\bar r)\sim 1/\bar r$; and (right) the number densities $n^\pm(\bar r)$ and absolute value of current density $|j_z(\bar r)|$, which decay like $1/\bar r^4$ (both with omitted multiplicative constant).}
\label{F:z-pinch-const}
\end{figure}

\begin{theorem}[$z$-pinch, variable drift equilibria]
\label{T:z-pinch-var}
Select $a>0$ (units of charge) and then let $d$ be the dimensionless constant \eqref{E:d-def}.  Let $\bar A_z(\bar r)$ solve the dimensionless second-order ODE
\begin{equation}\label{E:z-var-ODE}
-\partial_{\bar r}^2 \, \bar A_z - \frac1{\bar r} \, \partial_{\bar r} \, \bar A_z = - \tfrac12 \,d \,\bar A_z \,\exp[-\tfrac12\bar A_z^2]
\end{equation}
with boundary conditions at $\bar r=0$ given by
\begin{equation}\label{E:z-var-boundary}
\bar A_z(0) \text{ finite} \,, \qquad \partial_{\bar r}\bar A_z(0) =0
\end{equation}
Then we can produce an equilibrium solution to the VM system as follows.  Let $r = \beta \bar r$, the electric potential $\phi=0$, the cylindrical components of the magnetic potential $\bds A$ be
$$ A_r=0\,, \quad A_\theta=0\,, \quad A_z(r) = \frac{\sqrt 2 \bar A_z(\bar r)}{\lambda} \,, $$
and the densities be $f^\pm (r, v_r,v_\theta,v_z) = \mu^\pm( e^\pm(v_r,v_\theta,v_z),p_z^\pm(r,v_z))$, where
\begin{equation}\label{E:modz-mu-pm-2}
\mu^\pm(e^\pm, p_z^\pm) =  \frac{a}{\beta^2\zeta (\sigma^\pm)^3|q^\pm|} \exp \left( -\frac{e^\pm}{m^\pm (\sigma^\pm)^2}  - \frac{ (p_z^\pm)^2}{2(m^\pm)^2 (\sigma^\pm)^2}\right)
\end{equation}
To get the explicit $r$ and $\bds v$ expression for the densities exhibiting the shift in the $z$-velocity, we complete the square in the exponential to obtain
$$\mu^\pm(r, v_r, v_\theta, v_z) =  \frac{n^\pm(\bar r)}{(2\pi)^{3/2} (\sigma^\pm)^3}   \exp \Big[ - \frac12 \Big( \frac{v_r}{\sigma^\pm} \Big)^2 - \frac12 \Big( \frac{v_\theta}{\sigma^\pm} \Big)^2 - \Big( \frac{ v_z}{\sigma^\pm} \pm \frac12  \bar A_z(\bar r) \Big)^2\Big]$$
where the number densities of the two species are
$$n^\pm( \bar r) =\frac{(2\pi)^{3/2}}{\beta^2\zeta} \frac{a}{|q^\pm|} \exp [ - \tfrac12 \bar A_z(\bar r)^2]$$
Then $(f^\pm, \phi, \bds A)$ is an equilibrium solution to VM.  The resulting charge density $\rho=0$ and the current has cylindrical components
\begin{equation}\label{E:z-var-curr}
j_r=0\,, \quad j_\theta=0\,, \quad j_z(r) = -\frac{d}{\sqrt 2 \mu_0 \beta^2 \lambda}  \bar A_z(\bar r)  \exp[ -\tfrac12 \bar A_z(\bar r)^2]
\end{equation}
The electric field $\bds E=0$ and the magnetic field $\bds B$ has cylindrical components 
$$
B_r=0\,, \quad B_z=0\,, \quad  B_\theta(r) = \frac{\sqrt 2}{\lambda \beta} \bar B_\theta(\bar r)  \,, \quad \text{where }\bar B_\theta = -\partial_{\bar r} \bar A_z$$
\end{theorem}
This is proved in \S\ref{SS:var-z}.  We remark there that an apparently more general ansatz for $\mu^\pm$ adding a term linear in $p_z^\pm$ can be effectively removed by a shift in $A_z$.  In \eqref{E:z-var-ODE}, the constant $d$ can be removed by replacing $\bar r$ by $\sqrt{d}\, \bar r$.   Note that, in constrast to \eqref{E:z-const-curr}, the equation \eqref{E:z-var-ODE} is invariant under the change $\bar A_z \to -\bar A_z$.  

Although in contrast to \eqref{E:z-const-ODE}, \eqref{E:z-var-ODE} admits solutions decaying exponentially as $\bar r\to +\infty$, the mechanical analysis at the end of \S\ref{SS:var-z} shows that none of these solutions will satisfy the boundary conditions \eqref{E:z-var-boundary} at $\bar r=0$.  Indeed, if they did, such solutions would violate the conclusion of \cite[Theorem 2, p. 275]{GlasseyStrauss1984}, which implies that there are no nontrivial finite-energy equilibrium solutions to VM in $\mathbb{R}^3$.   There are no solutions to \eqref{E:z-var-ODE} for which $\bar A_z(\bar r)$ converges to a nonzero constant as $\bar r\to+\infty$, and thus the only possible behavior for solutions that satisfy \eqref{E:z-var-boundary} is 
$$\bar A_z(\bar r) \to \infty \quad \text{ as } \quad \bar r \to +\infty$$
This gives the physically desirable consequence that $j_z(r) \to 0$ as $r \to \infty$, from \eqref{E:z-var-curr}.  From \eqref{E:z-var-ODE}, we have that as $\bar r\to \infty$, $\bar A_z \sim A$ satisfying $-\partial_{\bar r}^2 \, A - \frac1{\bar r} \, \partial_{\bar r} \, A = 0$, and thus
\begin{equation}\label{E:z-var-j-decay}
\bar A_z(\bar z) \sim  C\ln \bar r + E \quad \text{ as } \quad \bar r \to +\infty
\end{equation}
For example, simulations with $\bar A_z(0)=2$ suggest that $C\approx 1.67$ in that case, see Figure \ref{F:z-pinch-var}.  The current given by \eqref{E:z-var-curr} behaves as
$$j_z(r) \sim (\log \bar r) \bar r^{-\frac{C}{2}\log \bar r} \quad \text{ as } \bar r \to +\infty$$
thus, decaying faster than any power, in contrast to the $\bar r^{-4}$ decay rate for solutions in Theorem \ref{T:z-pinch-const}. We thus have a super-power rate of decay for the number densities
\begin{equation}
\label{E:z-var-n-decay}
n^\pm(\bar r) \sim  \bar r^{-\frac{C}{2}\log \bar r} \quad \text{ as } \bar r \to +\infty
\end{equation}
and the $z$-drift velocity is $\mp \bar A_z(\bar r) \sigma^\pm /2$, which is dependent on $\bar r$, in contrast to Theorem \ref{T:z-pinch-const}.  Thus we have called this the $z$-pinch equilibrium with variable drift velocity.  Notably, due to the logarithmic divergence of $\bar A_z(\bar r)$ as $\bar r\to \infty$, this drift speed diverges as $\bar r\to \infty$.  The drift velocities of the two species are in opposite directions.

Note that in the exponential of the expression for $\mu^\pm$ above, there is a coefficient of $\frac12$ in front of the $v_r$ and $v_\theta$ terms, but a coefficient of $1$ in front of the $v_z$ term.  The physical interpretation is that the thermal speeds are different in the $r$, $\theta$ directions versus the $z$ direction.  In the physics literature, this was studied by Mahajan \cite[Equation (37)]{mahajan_1989} and by Channon \& Coppins \cite[Equation (3.5)]{CHANNON_COPPINS_2001} by an approach different from ours.

\begin{figure}
\includegraphics[scale=0.75]{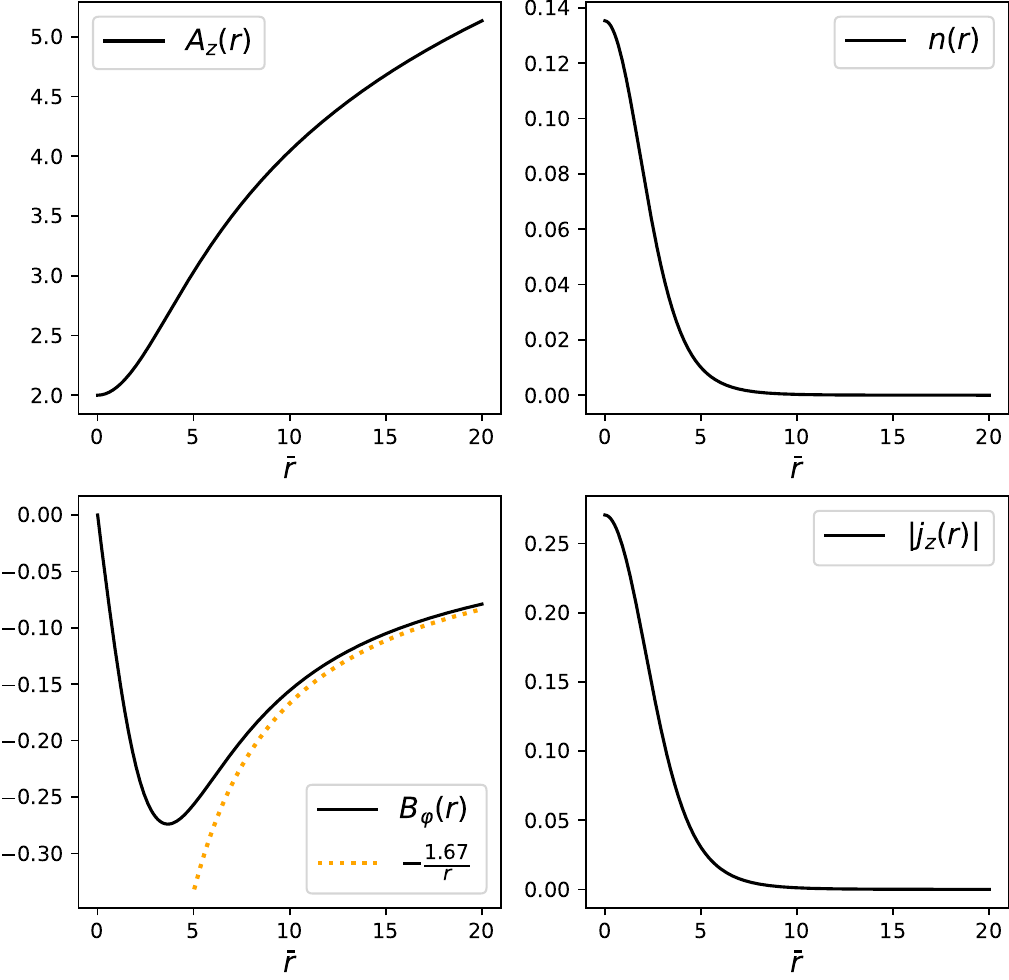}
\caption{(top left) plot of $\bar A_z(\bar r)$ solving \eqref{E:z-var-ODE} in Theorem \ref{T:z-pinch-var} with $\bar A_z(0)=2$, $\partial_{\bar r} \bar A_z(0)=0$, which diverges as $\approx 1.67 \log \bar r$ as $\bar r \to +\infty$; (bottom left) plot of the corresponding magnetic field $\bar B_\theta(\bar r)=-\partial_{\bar r}\bar A_z(\bar r)$, which decays slowly as $\approx -1.67/\bar r$; (top right) the number densities $n^\pm(\bar r)$ (with omitted multiplicative constant), decaying at the super-power rate \eqref{E:z-var-n-decay}; (bottom right) $|j_z(r)|$  (with omitted multiplicative constant), decaying at the super-power rate \eqref{E:z-var-j-decay}.}
\label{F:z-pinch-var}
\end{figure}

\begin{theorem}[$\theta$-pinch equilibria]
\label{T:theta-pinch-var}
Select $a>0$ (units of charge) and then let $d$ be the dimensionless constant \eqref{E:d-def}.
Let $\bar A_\theta(\bar r)$ solve the dimensionless second-order ODE
\begin{equation}\label{E:theta-var-ODE}
\Big(- \partial_{\bar r}^2 - \frac{1}{\bar r} \partial_{\bar r} + \frac{1}{\bar r^2} \Big)\bar A_\theta = - d \frac{ \bar r^2 \bar A_\theta}{(1 + \bar r^2)^{3/2}}\exp\left[ - \frac{ \bar r^2 \bar A_\theta^2}{2(1 + \bar r^2)} \right]
\end{equation}
with the physical boundary conditions 
\begin{equation}\label{E:t-var-boundary}
\bar A_\theta(0)=0 \,, \qquad \partial_r A_\theta(0) \text{ finite }
\end{equation}
Then we can produce an equilibrium solution to the VM system as follows.   
Let $r=\beta \, \bar r$ and take the electric potential $\phi=0$, the cylindrical components of the magnetic potential $\bds A$ to be 
$$A_r=0 \,, \quad A_\theta(r) = \lambda^{-1} \bar A_\theta(\bar r) \,, \quad A_z=0\,, $$
and the densities to be $f^\pm (r, v_r, v_\theta, v_z) = \mu^\pm( e^\pm(v_r,v_\theta,v_z),\ell_z^\pm(r, v_\theta))$, where
\begin{equation}\label{E:modt-mu-pm-2}
\mu^\pm(e^\pm, \ell_z^\pm) =  \frac{a}{\beta^2\zeta (\sigma^\pm)^3|q^\pm|} \exp \left( -\frac{ e^\pm}{m^\pm (\sigma^\pm)^2}  - \frac{ (\ell_z^\pm)^2}{2(m^\pm)^2 (\sigma^\pm)^2 \beta^2}\right)
\end{equation}
From this, we obtain the expression for densities in terms of $r$ and $\bds v$ exhibiting the shift in $\theta$-velocity, by completing the square in the exponential of \eqref{E:modt-mu-pm-2},
\begin{equation}
\label{E:modt-mu-pm-3}
\begin{aligned}
\mu^\pm(r,v_r,v_\theta,v_z) =  n^\pm(\bar r)  \,&\, \frac{1}{(2\pi)^{3/2} (\sigma^\pm)^3}  \sqrt{1+\bar r^2} \exp\Big[ -\frac12 \Big(\frac{ v_r}{\sigma^\pm} \Big)^2 \\
&-\frac12 \Big(\frac{ v_z}{\sigma^\pm} \Big)^2 -\frac12(1+ \bar r^2) \Big(\frac{ v_\theta}{\sigma^\pm} \pm \frac{\bar r^2}{1+\bar r^2} \bar A_\theta(\bar r)\Big)^2 \Big]
\end{aligned}
\end{equation}
where the number densities 
$$n^\pm(\bar r) = \frac{(2\pi)^{3/2} a}{|q^\pm| \beta^2 \zeta} \frac{1}{\sqrt{1+\bar r^2}} \exp \Big[ -\frac{\bar r^2 \bar A_\theta(\bar r)^2}{2(1+\bar r^2)} \Big]$$

Then $(f^\pm, \phi, \bds A)$ is an equilibrium solution to the VM system.  The resulting total charge density $\rho=0$ and the total current has cylindrical components 
\begin{equation}
\label{E:t-var-curr}
j_r=0\,, \quad 
j_\theta(r)  = - \frac{d}{\mu_0\beta^2\lambda} \frac{ \bar r^2  \bar A_\theta(
\bar r)}{(1 + \bar r^2)^{3/2}} \exp\left[ - \frac{ \bar r^2 \bar A_\theta(\bar r)^2}{2(1 + \bar r^2)} \right]\,, \quad j_z=0\,,
\end{equation}
The electric field $\bds E=0$ and the magnetic field $\bds B$ has cylindrical components 
$$
B_r=0\,, \quad B_\theta=0\,, \quad B_z(r) = \frac{ \bar B_z(\bar r)}{\lambda \beta}  \,,$$
where
$$\bar B_z = \bar r^{-1} \partial_{\bar r}(\bar r \bar A_\theta) = \partial_{\bar r} \bar A_\theta + \frac{\bar A_\theta}{\bar r}$$
\end{theorem}
This is proved in \S\ref{SS:t-var}.  
Since ODE \eqref{E:theta-var-ODE} is invariant under $\bar A_\theta \mapsto -\bar A_\theta$, we might as well assume $\partial_{\bar r} \bar A_\theta(0)>0$.   In Figure \ref{F:theta-pinch-var}, we take $\bar A_\theta(0)=0$ and $\partial_{\bar r} \bar A_\theta(0)=1$.  Solutions satisfying the boundary conditions diverge as $\bar r \to \infty$, causing the right-side of \eqref{E:theta-var-ODE} to go to zero.  Thus $\bar A_\theta(\bar r)$ is approximately $A$ solving $(- \partial_{\bar r}^2 - \frac{1}{\bar r} \partial_{\bar r} + \frac{1}{\bar r^2})A = 0$, which has solutions $A(\bar r) = \bar r$ and $A(\bar r)=1/\bar r$.  Thus 
$$\bar A_\theta (\bar r) \approx C \bar r \quad \text{ as } \quad  \bar r \to \infty$$
In the solution computed for Figure \ref{F:theta-pinch-var} with $\bar A_\theta(0)=0$ and $\partial_{\bar r}\bar A_\theta(0)=1$, we find $C\approx 0.76$.

Since $\bar A_\theta(\bar r) \sim \bar r$ as $\bar r\to \infty$, the number densities behave as
$$n^\pm(r) \sim \bar r^{-1} \exp[- C^2 \bar r^2/2] \quad \text{ as } \quad \bar r \to \infty$$
Likewise from \eqref{E:t-var-curr}, we find
$$j_\theta(r) \sim  \exp[- C^2 \bar r^2/2] \quad \text{ as } \quad \bar r \to \infty$$
Thus, despite the fact that the magnetic field does not decay as $\bar r \to \infty$, the particles and current are strongly confined.  Note that from \eqref{E:modt-mu-pm-3}, there is 
$$\text{drift in the $\theta$-velocity} = \mp \sigma^\pm \frac{\bar r^2}{1+\bar r^2} \bar A_\theta(\bar r)$$
which is in the opposite direction for two $\pm$ species.  This drift \emph{accelerates} while the number of particles rapidly vanishes as $\bar r \to \infty$.  Since the current points in the $\theta$ direction, this solution is called a $\theta$-pinch. 

\begin{figure}
\includegraphics[scale=0.75]{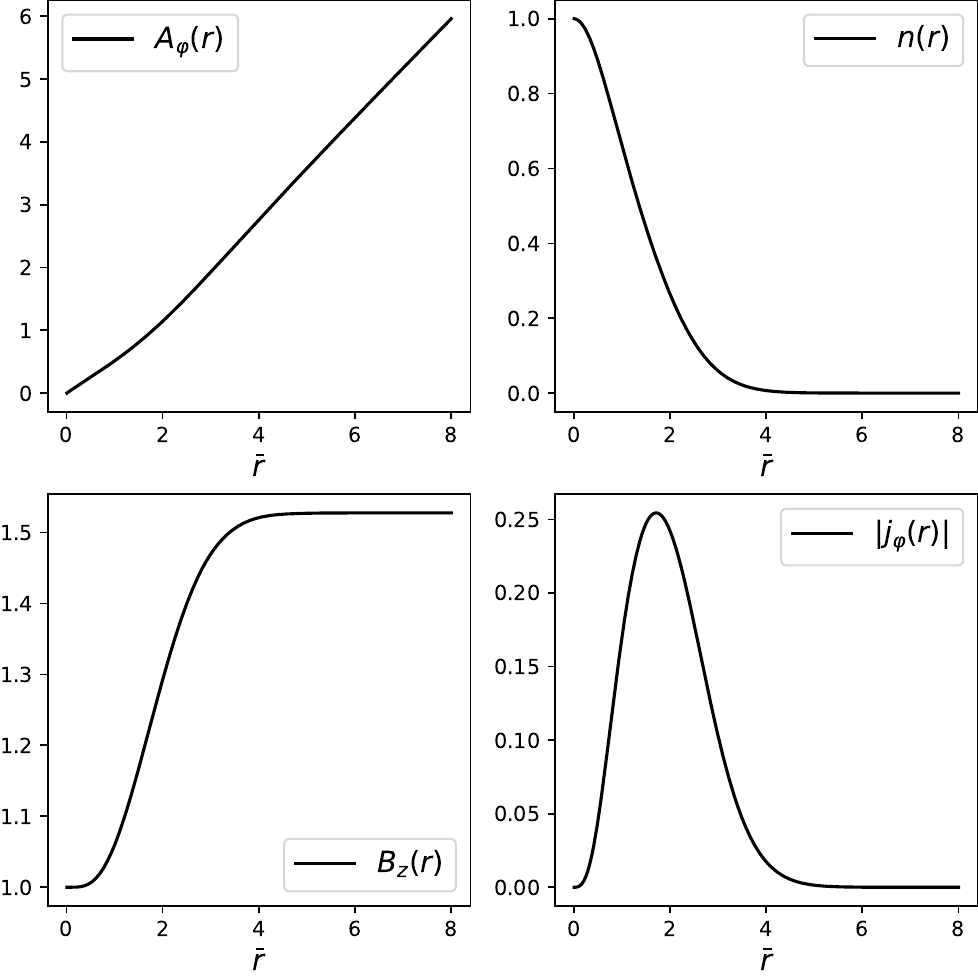}
\caption{(top left) plot of $\bar A_\theta(\bar r)$ solving \eqref{E:theta-var-ODE} in Theorem \ref{T:theta-pinch-var} with $\bar A_\theta(0)=0$ and $\partial_{\bar r}\bar A_\theta(0) = 1$, which grows as $\sim \bar r$ as $\bar r\to \infty$; (bottom left) plot of the corresponding magnetic field $\bar B_z(\bar r)=\bar r^{-1} \partial_{\bar r}(\bar r\bar A_\theta)(\bar r)$, which increases and levels off to a constant; (top right) plot of the number densities $n^\pm(\bar r)$ and (bottom right) $|j_\theta(r)|$, with omitted multiplicative constant, which decay at a near Gaussian rate, showing the strong localization of particles.}
\label{F:theta-pinch-var}
\end{figure}

We next discuss \emph{screw pinch} solutions that effectively interpolates between $z$-pinch and $\theta$-pinch.  Fix a pitch parameter $k \in \zeta \mathbb{Z}/2\pi$, which has SI units of m.  Using $$\delta(r) \defeq \sqrt{r^2+k^2}$$ (again SI units of m) rotate the cylindrical coordinate basis from $(\bds e_\theta, \bds e_z)$ to $(\bds e_\eta, \bds e_\sigma)$ where
$$\delta \bds e_\eta = r \bds e_\theta + k \bds e_z \,, \qquad \delta \bds e_\sigma = -k \bds e_\theta + r \bds e_z$$
The representations 
$$\bds v = v_r \bds e_r + v_\theta \bds e_\theta + v_z \bds e_z  = v_r \bds v_r + v_\eta \bds e_\eta + v_\sigma \bds e_\sigma$$
yield the component relationships
$$\delta v_\eta = r v_\theta + k v_z \,, \qquad \delta v_\sigma = -k v_\theta + r v_z$$
Analogous statements hold for $\bds A$ and $\bds j$.  For the equilibria that we present in Theorem \ref{T:screw-pinch-var} below, the densities $\mu^\pm$ are functions of $\ell_z^\pm + k p_z^\pm$, which depends only on $v_\eta$ and $A_\eta$.  Since the transformation $(\bds e_\theta, \bds e_z)$ to $(\bds e_\eta, \bds e_\sigma)$ is orthogonal, $e^\pm = v_r^2 + v_\eta^2 + v_\sigma^2$ when $\phi=0$.  
Consequently, the current density $\bds j$ is parallel to $\bds e_\eta$ (i.e. $j_\sigma=0$), and in this case one finds that $A_\sigma$ can be expressed in terms of $A_\eta$ via
\begin{equation}\label{E:Ase}
\partial_r (r\delta^{-1} A_\sigma) = 2kr \delta^{-3} A_\eta \qquad (\text{when }j_\sigma=0)
\end{equation}
This allows the Amp\`ere equation to be expressed entirely in terms of a second-order ODE connecting $A_\eta$ and $j_\eta$:
\begin{equation}\label{E:sp-Ampere}
 -\partial_{ r}^2 A_\eta - \frac{1}{ r}\partial_{ r}  A_\eta  + \Big( \frac{3 r^2}{\delta^4} - \frac{2}{\delta^2}\Big) A_\eta= \mu_0 j_\eta
\end{equation}
Also, we find that $\bds B$ is parallel to $\bds e_\sigma$ ($B_\eta=0$) and
\begin{equation}\label{E:sp-Bsigma}
B_\sigma = \frac{r}{\delta^2} A_\eta + \partial_r A_\eta
\end{equation}
When $k=0$, $A_\eta=A_\theta$, $j_\eta=j_\theta$ and $B_\sigma =B_z$. On a formal level, as $k\to \infty$, $A_\eta=A_z$, $j_\eta=j_z$, and $B_\sigma=-B_\theta$.  Both \eqref{E:sp-Ampere} and \eqref{E:sp-Bsigma} are consistent with these limits, reducing to the $\theta$-pinch when $k=0$ and $z$-pinch when $k\to \infty$.

\begin{theorem}[screw-pinch equilibria]
\label{T:screw-pinch-var}
Select $a>0$ (units of charge) and then let $d$ be the dimensionless constant \eqref{E:d-def}.
Let $\bar A_\eta(\bar r)$ solve
\begin{equation}\label{E:s-var-ODE}
 -\partial_{\bar r}^2\bar A_\eta - \frac{1}{\bar r}\partial_{\bar r}\bar  A_\eta  + \Big( \frac{3\bar r^2}{\delta^4} - \frac{2}{\delta^2}\Big)\bar A_\eta =  - d \, \frac{  \delta^2 \bar A_\eta}{ (1+\delta^2)^{3/2}} \exp\left[ - \frac{\delta^2 \bar A_\eta^2}{2(1+\delta^2)} \right]
\end{equation}
Then we can produce an equilibrium solution to VM as follows.   
Let $r = \beta \bar r$ and take the electric potential to be $\phi=0$, the magnetic potential $\bds A$ to have components $A_r=0$, 
$$A_\eta(r) = \lambda^{-1} \bar A_\eta(\bar r) \,,$$
and $A_\sigma$ given by \eqref{E:Ase};
and the densities to be 
$$f^\pm (r, v_r,v_\eta,v_\sigma) = \mu^\pm( e^\pm(v_r,v_\eta,v_\sigma),(\ell_z^\pm+kp_z^\pm)(r,v_\eta))\,,$$ 
where
\begin{equation}\label{E:mods-mu-pm-2}
\mu^\pm(e^\pm, \ell_z^\pm+kp_z^\pm) =  \frac{a}{\beta^2\zeta (\sigma^\pm)^3|q^\pm|} \exp \left( -\frac{e^\pm}{m^\pm (\sigma^\pm)^2}  - \frac{ (\ell_z^\pm+kp_z^\pm)^2}{2(m^\pm)^2 (\sigma^\pm)^2 \beta^2}\right)
\end{equation}
By completing the square in the exponential of \eqref{E:mods-mu-pm-2},
\begin{equation}
\label{E:mods-mu-pm-3}
\begin{aligned}
\mu^\pm(r,v_r,v_\eta,v_\sigma) =  n^\pm(\bar r) \, & \, \frac{1}{(2\pi)^{3/2} (\sigma^\pm)^3}  \sqrt{1+\delta^2} \exp\Big[ -\frac12 \Big(\frac{ v_r}{\sigma^\pm} \Big)^2 \\
&-\frac12 \Big(\frac{ v_\sigma}{\sigma^\pm} \Big)^2 -\frac12(1+ \delta^2) \Big(\frac{ v_\eta}{\sigma^\pm} + \frac{\delta^2}{1+\delta^2} \bar A_\eta(\bar r)\Big)^2 \Big]
\end{aligned}
\end{equation}
where the number densities are
$$n^\pm(\bar r) = \frac{(2\pi)^{3/2} a}{|q^\pm| \beta^2 \zeta} \frac{1}{\sqrt{1+\delta^2}} \exp \Big[ -\frac{\bar r^2 \bar A_\eta(\delta)^2}{2(1+\delta^2)} \Big]$$

Then $(f^\pm, \phi, \bds A)$ is an equilibrium solution to VM.  The resulting charge density $\rho=0$ and the current $\bds j$ has components 
$$j_r=0 \,, \quad j_\sigma=0 \,, \quad 
 j_\eta(r) = - \frac{d}{\mu_0\beta^2} \, \frac{  \delta^2 \bar A_\eta}{ (1+\delta^2)^{3/2}} \exp\left[ - \frac{\delta^2 \bar A_\eta^2}{2(1+\delta^2)} \right]
$$
\end{theorem}

This is proved in \S\ref{SS:s-var}.
For $k=0$, the equation reduces to the $\theta$-pinch equation of Theorem \ref{T:theta-pinch-var}, with the $\bar r=0$ boundary conditions stated there.  But when $k\neq 0$, the boundary condition at $r=0$ is
$$\bar A_\eta(0) \text{ finite} \,, \qquad \partial_r \bar A_\eta(0) = 0$$
Moreover, as $r\to +\infty$, 
$$\frac{3}{\delta^4}-\frac{2}{\delta^2} = \frac{1}{r^2} - \frac{4k^2}{r^4} + O(r^{-6})$$
so the solutions are $\bar A_\eta(\bar r) \sim \bar r$ and $\bar A_\eta(\bar r)\sim r^{-1}$ as $\bar r \to +\infty$.    Note that \eqref{E:s-var-ODE} is invariant under $\bar A_\eta \to -\bar A_\eta$ and thus it suffices to assume $\bar A_\eta(0)>0$.   See Figure \ref{F:s-pinch-var} for a plot of the numerically obtained solution for $k=1$ and $\bar A_\eta(0) = 0.2$.


Since $\bar A_\theta(\bar r) \sim \bar r$ as $\bar r\to \infty$, the number densities behave as
$$n^\pm(r) \sim \bar r^{-1} \exp[- C^2 \bar r^2/2] \quad \text{ as } \quad \bar r \to \infty$$
Likewise from \eqref{E:t-var-curr}, we find
$$j_\theta(r) \sim  \exp[- C^2 \bar r^2/2] \quad \text{ as } \quad \bar r \to \infty$$
Thus, despite the fact that the magnetic field does not decay as $\bar r \to \infty$, the particles and current are strongly confined.  Note that from \eqref{E:mods-mu-pm-3}, there is 
$$\text{drift in the $\eta$-velocity} = \pm \sigma^\pm \frac{\delta^2}{1+\delta^2} \bar A_\eta(\bar r)$$
which is in the opposite direction for two $\pm$ species.  This drift \emph{accelerates} while the number of particles rapidly vanishes as $\bar r \to \infty$.  


\begin{figure}
\begin{center}
\includegraphics[scale=0.75]{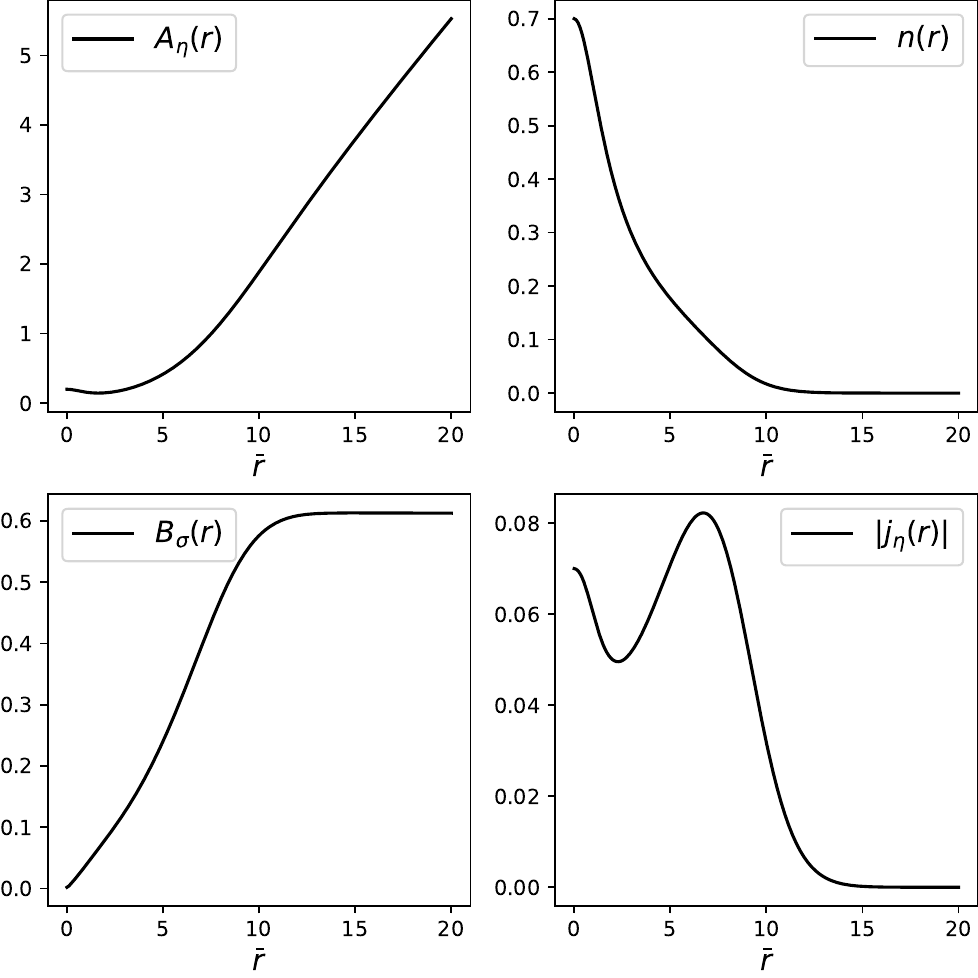}
\end{center}
\caption{(top left) plot of $\bar A_\eta(\bar r)$ solving \eqref{E:s-var-ODE} for pitch $k=1$ and initial conditions $\bar A_\eta(0)=0.2$, $\partial_{\bar r} \bar A_\eta(0) =0$; (bottom left) the corresponding magnetic field, which increases and levels off to a nonzero constant; (top right) the number density and (bottom right) the current density, both of which decay at a Gaussian rate, although the shapes are distorted in comparison to Figure \ref{F:theta-pinch-var} by the screw-pinch geometry.}
\label{F:s-pinch-var}
\end{figure}

The paper is organized as follows.  In \S\ref{S:VM}, we state the Vlasov-Maxwell system, introduce the electric and magnetic potentials, and in \S\ref{S:cyl}, we review the conventions and standard vector calculus formulae in cylindrical coordinates, and give a couple of examples highlighting computational pitfalls.  In \S\ref{S:equil}, we discuss the Noether invariants for particle trajectories that correspond to $\theta$ invariance (conservation of $\ell_z^\pm$) and $z$-invariance (conservation of $p_z^\pm$).  The densities of equilibrium solutions are then given as functions of these invariants.  In \S\ref{S:z-pinch}, we discuss $z$-pinch equilibria and give the proof of Theorems \ref{T:z-pinch-const}, \ref{T:z-pinch-var}.  In \S\ref{S:theta-pinch}, we discuss $\theta$-pinch equilibria and give the proof of Theorem \ref{T:theta-pinch-var}.   In \S\ref{S:screw}, we discuss screw-pinch equilibria and give the proof of Theorem \ref{T:screw-pinch-var}.  In \S\ref{S:external}, we discuss the modifications needed to include an external field, in which case it is possible to produce solutions in the context of variable drift $z$-pinch, $\theta$-pinch, or screw-pinch for which the magnetic field converges exponentially fast to zero.  In \S\ref{S:relativistic}, we briefly describe how the form of the equilibria can be adapted to the relativistic case.

In future work, we intend to study the linear stability and instability of the equilibria that we have presented in Theorems \ref{T:z-pinch-const}--\ref{T:screw-pinch-var} using the framework developed by Lin \& Strauss \cite{LinStraussLinear2007, LinStrauss2008}, and further explained schematically by Lin \cite{Lin2022} using the separable Hamiltonian language of Lin \& Zeng \cite{LinZeng2022}.  Other results exploiting these ideas are Nguyen \& Stauss \cite{NguyenStrauss2013, NguyenStrauss2014} and K.Z. Zhang \cite{KZZhang2019}.  The methods cited above reduce the linear stability of the equilibria to a spectral condition for an associated second-order scalar problem, which is itself still nontrivial and not yet investigated thoroughly for explicit or semi-explicit equilibria such as those we obtain here.  Stability of the $z$-pinch was studied in the MHD setting by Bian, Guo \& Tice \cite{BianGuoTice_Inviscid, BianGuoTice_Viscous}.  

In this paper, we have considered equilibria in cylindrical coordinates that depend only on the radial variable and for which the field equations reduce to a second-order ODE.  To obtain more complex equilibria that depend on two or all three variables, more abstract (less explicit) methods are needed -- for example, bifurcations methods are employed by K.Z. Zhang \cite{KZZhang2023}, and that paper gives an account of related literature.  However, solutions such as those we obtain can be the starting point in a bifurcation analysis.  Our purely radial solution should have an analogue in the more physical settting of the solid torus in the torodial/poloidal coordinate system.  Linear stability of RVM equilibria in this setting was studied by Nguyen \& Strauss \cite{NguyenStrauss2014}.  Non-symmetric plasma equilibria for the magnetohydrostatic (steady three-dimensional incompressible Euler) equations with a small force are constructed in Constantin, Drivas \& Ginsberg \cite{constantin2021quasisymmetric}. 

\subsection{Acknowledgments}
J.H. was partially supported by NSF Grant DMS-2452781.  K.Z.Z. was partially supported by Simons Travel Support for Mathematicians SFI-MPS-TSM-00012207.

\section{Vlasov-Maxwell equation \label{S:VM}}

Fix a vertical length $\zeta>0$.
We consider the $2$-species Vlasov-Maxwell system in $(\bds x, \bds v)$, where $\bds x\in \Omega=\mathbb{R}^2\times [0,\zeta)$, with periodic boundary conditions in $z$, and $\bds v \in \mathbb{R}^3$.  This system includes the Vlasov equations for the densities $f^\pm(t, \bds x, \bds v)$:
\begin{equation}
\label{E:Vlasov}
\partial_t f^{\pm} + \bds v\cdot \nabla_{\bds x} f^\pm + \frac{q^\pm}{m^\pm}(\bds E + \bds v \times \bds B) \cdot \nabla_{\bds v} f^\pm = 0
\end{equation}
and the Maxwell equations for the electric field $\bds E(t, \bds x)$ and magnetic field $\bds B(t, \bds x)$:
\begin{equation}
\label{E:Maxwell}
\begin{aligned}
&\nabla\cdot \bds E = \epsilon_0^{-1}\rho &\qquad& \nabla \times \bds E = -\partial_t \bds B \\
&\nabla \cdot \bds B = 0 && \nabla \times \bds B = \mu_0(\bds j + \epsilon_0 \partial_t \bds E)
\end{aligned}
\end{equation}
Here, the charge density $\rho$ and current density $\bds j$ are
\begin{equation}\label{E:charge_current}
\rho(t,\bds x) = \sum_{\pm} q^\pm \int f^\pm(t,\bds x, \bds v ) \, d\bds v \,, \qquad \bds j(t,\bds x) = \sum_{\pm} q^\pm \int \bds v f^\pm(t,\bds x, \bds v) \, d\bds v
\end{equation}

As is customary, we introduce the scalar electric potential $\phi(t,\bds x)$ and the vector magnetic potential $\bds A(t,\bds x)$ satisfying
\begin{equation}\label{E:potentials}
\bds E = -\nabla \phi - \partial_t \bds A \,, \qquad \bds B = \nabla \times \bds A 
\end{equation}
together with the Coulomb gauge condition $\nabla \cdot \bds A = 0$. Using the standard formula $\nabla \times (\nabla \times \bds A) = \nabla( \nabla \cdot \bds A) - \Delta \bds A$, the field equations convert to
\begin{equation}
\label{E:Maxwell-pot1}
-\Delta \phi = \epsilon_0^{-1} \rho \,, \qquad -\Delta \bds A = \mu_0 \bds j - \mu_0\epsilon_0 \partial_t \nabla \phi - \mu_0 \epsilon_0 \partial_t^2 \bds A
\end{equation}

\section{Cylindrical coordinates \label{S:cyl}}
At every $\bds x\in \Omega = \mathbb{R}^2\times [0,\zeta)$, alternatively represented in cyindrical coordinates as $(r,\theta,z)$, there is an orthonormal basis of the tangent space spanned by
$$ \bds e_r = (\cos \theta, \sin \theta, 0) \,,\quad 
\bds e_\theta = (-\sin \theta, \cos\theta, 0) \,, \quad \bds e_z = (0,0,1)$$
A position-velocity pair $(\bds x, \bds v)$ is regarded as belonging to the tangent bundle, so $\bds v \in T_{\bds x}\Omega$ and velocity can be represented 
$$\bds v = v_r \, \bds e_r + v_\theta \, \bds e_\theta + v_z \, \bds e_z$$
Vector fields, e.g.\ $A$, are sections of $T\Omega$, so $A(\bds x) \in T_{\bds x}\Omega$ can be decomposed into components which define the coefficient functions $A_r$, $A_\theta$, and $A_z$:
$$\bds A = A_r \bds e_r + A_\theta \bds e_\theta + A_z \bds e_z$$
The divergence, curl, and Laplacian have representations with respect to this frame.  For a scalar function, we have the formulae
\begin{equation}\label{E:grad-cyl}
\nabla \phi = \partial_r \phi \, \bds e_r + \frac{1}{r} \partial_\theta \phi \, \bds e_\theta + \partial_z\phi \,\bds e_z
\end{equation}
\begin{equation}
\label{E:lap-cyl}
\Delta \phi = \frac1r\partial_r (r\partial_r \phi) + \frac1{r^2} \partial_\theta^2\phi + \partial_z^2 \phi
\end{equation}
For a vector field, we have the formulae
\begin{equation}
\label{E:curl-cyl}
\nabla \times \bds A = (\frac1r \partial_\theta A_z - \partial_z A_\theta) \bds e_r + (\partial_z A_r - \partial_r A_z) \bds e_\theta + \frac1r(\partial_r (rA_\theta) - \partial_\theta A_r) \bds e_z
\end{equation}
\begin{equation}
\label{E:div-cyl}
\nabla \cdot \bds A = \frac1r \, \partial_r  \, r \, A_r  + \frac1r \partial_\theta A_\theta + \partial_z A_z
\end{equation}
\begin{equation}
\label{E:vec-lap-cyl}
\Delta \bds A = (\Delta A_r - \frac{A_r}{r^2} - \frac{2}{r^2}\partial_\theta A_\theta) \bds e_r + (\Delta A_\theta - \frac{A_\theta}{r^2} + \frac{2}{r^2} \partial_\theta A_r) \bds e_\theta + \Delta A_z \bds e_z
\end{equation}

There are many pitfalls when working with cylindrical coordinates, stemming from the fact that the coordinate frame depends on $\theta$.  We highlight a few examples. 

\begin{example}
$\nabla_{\bds v}$ or $\nabla_{\bds p}$ are not computed using the above formulae for the gradient. The above formulae for the gradient are only used for computing $\nabla_{\bds x}$.  For example, if $H$ is a scalar function, 
$$\nabla_{\bds p} H = \partial_{p_r} H \bds e_r + \partial_{p_\theta} H \bds e_\theta + \partial_{p_z}H \bds e_z$$
\end{example}

\begin{example}
Suppose that the components of the vector field $\bds A$ are independent of $\theta$.  This does not imply that the value of the vector field $\bds A$ as an element of $\mathbb{R}^3$ is independent of $\theta$.  Indeed, using $\bds A = A_r \bds e_r + A_\theta \bds e_\theta + A_z \bds e_z$ and the assumption that $\partial_\theta A_r=0$, $\partial_\theta A_\theta=0$, $\partial_\theta A_z=0$ together with $\partial_\theta \bds e_r = \bds e_\theta$, $\partial_\theta \bds e_\theta = -\bds e_r$, we obtain
$$\partial_\theta \bds A = - A_\theta \bds e_r + A_r \bds e_\theta  $$
Thus we can still have $\partial_\theta \bds A \neq \bds 0$ in this situation.  This is geometrically clear -- as we walk around the cylinder at fixed radius, the direction of $\bds A$ has to change in accordance with the cylindrical basis.  For example, if $\bds A$ is tangent to the cylinder, then to remain tangent as we walk, the direction of $\bds A$ must change.
\end{example}

\begin{example}
When we decompose $\bds v = v_r \bds e_r + v_\theta \bds e_\theta + v_z \bds e_z$,  both $v_r$ and $v_\theta$ depend on $\theta$ and we have
$$\partial_\theta v_r = v_\theta \,, \qquad \partial_\theta v_\theta = -v_r$$
Suppose that $f(r,z, v_r,v_\theta,v_z)$ has no explicit dependence on $\theta$, but as a function of $v_r$ and $v_\theta$, has \emph{implicit} dependence on $\theta$.  The notation $\partial_\theta f$ is ambiguous -- if it refers to the explicit dependence on $\theta$, it is zero, but if it accounts for the implicit dependence, it is
\begin{equation}
\label{E:f-phi}
\partial_\theta [f (r,z, v_r, v_\theta,v_z)] = \partial_{v_r} f  \frac{\partial v_r}{\partial \theta} + \partial_{v_\theta} f \frac{\partial v_\theta}{\partial \theta}  = v_\theta \partial_{v_r} f - v_r \partial_{v_\theta} f
\end{equation}
In the case when $f$ has no explicit dependence on $\theta$, $z$, the charge density $\rho$  and the \emph{components} $j_r$ , $j_\theta$, $j_z$ of the current density depend only on $r$.  We have the important formula for the time rate of change of charge:
\begin{equation}
\label{E:dt_charge}
\partial_t \rho  = - \frac1r \, \partial_r \, r \, j_r 
\end{equation}
In order to correctly derive this directly in cylindrical coordinates, one needs to use \eqref{E:grad-cyl} to obtain
$$\bds v \cdot \nabla_x f = v_r \partial_r f + \frac{v_\theta}{r} \partial_\theta f$$
and substitute \eqref{E:f-phi}.  The equation \eqref{E:dt_charge} can also be derived by using that $\partial_t \rho = - \nabla_{\bds x} \cdot \bds j$ (which is most easily derived in Euclidean coordinates) and \emph{then} passing to cylindrical coordinates using \eqref{E:div-cyl}.
\end{example}

\begin{example}
In the Hamiltonian setting, the vector $\bds p$ is independent of the vector $\bds x$ and hence the Euclidean coordinates $p_x$, $p_y$, $p_z$  of the vector $\bds p$ are independent of $\theta$.  However, when decomposing $\bds p = p_r \bds e_r + p_\theta \bds e_\theta + p_z \bds e_z$, we have 
\begin{equation}
\label{E:dphi-p}
\partial_\theta p_r = p_\theta \text{ and } \partial_\theta p_\theta = -p_r\,.
\end{equation}
This follows from the fact that
$$p_r = p_x \cos \theta + p_y \sin \theta \,, \qquad p_\theta = -p_x\sin \theta + p_y \cos\theta $$
and the fact that $p_x$, $p_y$ are independent of $\theta$.  
\end{example}

In this paper, we assume the scalar field $\phi$ and the \emph{cylindrical components} $A_r$, $A_\theta$, and $A_z$ are independent of $\theta$ and $z$.   As emphasized above, this does not imply that the vector field $\bds A$ is independent of $\theta$, due to the movement of the frame.  We will call this the \emph{cylindrically and vertically symmetric} case, and in this setting, the formulae \eqref{E:grad-cyl}, \eqref{E:lap-cyl} simplify to:
$$\nabla \phi = \partial_r \phi \, \bds e_r \,, \qquad \Delta \phi = \frac1r\partial_r (r\partial_r \phi) $$
The formula \eqref{E:curl-cyl} simplifies to 
\begin{equation}\label{E:curlA}
\bds B = \nabla \times \bds A =  - \partial_r A_z \, \bds e_\theta + \frac1r\partial_r (rA_\theta) \,\bds e_z
\end{equation}
and thus automatically $B_r=0$.   We see that we might as well take $A_r=0$ since $A_r$ plays no role in the form of $\bds B$, and then the condition $\nabla \cdot \bds A=0$ follows from \eqref{E:div-cyl}.   Also, we obtain 
\begin{equation}
\label{E:ax_vert_BfromA}
E_r=-\partial_r \phi \,, \qquad B_\theta = -\partial_r A_z \,, \qquad B_z=  \frac{1}{r} \partial_r (r A_\theta)
\end{equation}
Finally, using \eqref{E:vec-lap-cyl}, we have (recall $A_r=0$)
\begin{align*}
\Delta \bds A &= (\frac1r \partial_r r \partial_r A_r - \frac{A_r}{r^2} ) \bds e_r + (\frac1r \partial_r r \partial_r A_\theta - \frac{A_\theta}{r^2} ) \bds e_\theta + (\frac1r \partial_r r \partial_r A_z) \bds e_z \\
&=    (\partial_r \frac1r \partial_r r  A_\theta)  \bds e_\theta + (\frac1r \partial_r r \partial_r A_z )\bds e_z
\end{align*}

Thus, in the cylindrically and vertically symmetric case, the VM equations \eqref{E:Vlasov}, \eqref{E:Maxwell}, \eqref{E:charge_current} take the form
\begin{equation}
\label{E:ax_vert_Maxwell}
\begin{aligned}
&E_r = \epsilon_0^{-1} \, \frac1r \, \partial_r^{-1} \, r \rho 
& \qquad & \\
& \partial_t A_\theta = - E_\theta  
& \qquad 
& \mu_0 \epsilon_0 \partial_t E_\theta = - \mu_0 j_\theta - \partial_r  \frac1r  \, \partial_r \, r \, A_\theta \\
& \partial_t A_z = -E_z 
& \qquad
& \mu_0 \epsilon_0 \partial_t E_z = - \mu_0 j_z - \frac1r \, \partial_r \, r \, \partial_r \, A_z
\end{aligned}
\end{equation}
We have removed the equation $\mu_0 \epsilon_0 \partial_t E_r = - \mu_0 j_r$ from the list since it can be derived from $E_r = \epsilon_0^{-1} \, \frac1r \, \partial_r^{-1} \, r \rho $ using \eqref{E:dt_charge}.  We can recover $B_\theta$, and $B_z$ from \eqref{E:ax_vert_BfromA}.

Also, the Vlasov transport equation \eqref{E:Vlasov}, in this context, using \eqref{E:grad-cyl} and $B_r=0$, takes the form
\begin{equation}
\label{E:ax_vert_Vlasov}
\begin{aligned}
0 = &\;\partial_t f^\pm + v_r \partial_r f^\pm + v_\theta \frac{1}{r} \partial_\theta f^\pm + \frac{q^\pm}{m^\pm}( E_r + v_\theta B_z - v_z B_\theta) \partial_{v_r} f^\pm \\
&+ \frac{q^\pm}{m^\pm} ( E_\theta - v_r B_z) \partial_{v_\theta} f^\pm + \frac{q^\pm}{m^\pm} (E_z + v_r B_\theta) \partial_{v_z} f^\pm 
\end{aligned}
\end{equation}
where we recall $\partial_\theta f^\pm$ is given by \eqref{E:f-phi} and we have the ability to rewrite $B_\theta$, $B_z$ in terms of $A_\theta$, $A_z$ via \eqref{E:ax_vert_BfromA}.

\section{General form for cylindrically and vertically symmetric equilibria \label{S:equil}}

Equilibria are studied using the particle trajectories, which are characteristics of the Vlasov components.  Let $\bds x^\pm(t)$ and $\bds v^\pm(t)$ be solutions of the ODEs describing classical motion of a particle with charge $q^\pm$ in the  electromagnetic field
$$\dot{\bds x}^\pm = \bds v^\pm \,, \qquad \dot{\bds v}^\pm = \frac{q^\pm}{m^\pm}( \bds E(\bds x^\pm) + \bds v^\pm \times \bds B(\bds x^\pm))$$
The energy and momentum are
$$e^\pm (\bds x^\pm, \bds v^\pm) = \frac12m^\pm |\bds v^\pm|^2 + q^\pm \phi(\bds x^\pm) \,, \qquad \bds p^\pm(\bds x^\pm, \bds v^\pm) = m^\pm \bds v^\pm + q^\pm \bds A(\bds x^\pm)$$
It can be checked that the energy $e^\pm(\bds x^\pm, \bds v^\pm)$ is conserved in time.  To obtain the Hamiltonian form, define the Hamiltonian $H^\pm(\bds x^\pm, \bds p^\pm)$ to be energy with $\bds v^\pm$ replaced in terms of $\bds p^\pm$
$$H^\pm (\bds x^\pm, \bds p^\pm) = \frac{|\bds p^\pm - q^\pm \bds A(\bds x^\pm)|^2}{2m^\pm} + q^\pm \phi(\bds x^\pm)$$
Then the ODEs become
$$\dot{\bds x}^\pm = \nabla_{\bds p}H^\pm(\bds x^\pm, \bds p^\pm) \,, \qquad \dot{\bds p}^\pm = -\nabla_{\bds x} H^\pm(\bds x^\pm, \bds p^\pm)$$
An immediate consequence is that if $\phi$ and $\bds A$ are independent of $z$, then $H^\pm(\bds x^\pm, \bds p^\pm)$ are independent of $z$, and thus $\dot p_z^\pm=0$, i.e. $p_z^\pm$ is conserved.

Define angular momentum $\bds \ell^\pm$ as 
$$\bds \ell^\pm(\bds x^\pm, \bds v^\pm) = \bds x^\pm \times \bds p^\pm$$
If $\phi$ and the components of $\bds A$, namely $A_r$, $A_\theta$, and $A_z$ are independent of $\theta$, then $\ell_z^\pm$ is conserved.  This is nicely explained by Noether's theorem, but also verified by direct computation.  Note that
$$\ell_z^\pm = r p_\theta^\pm = r (m^\pm v_\theta + q^\pm A_\theta(\bds x^\pm))$$

\begin{example}
We show that $\ell_z$ is conserved if $\phi$ and the cylindrical components $A_r$, $A_\theta$, and $A_z$ of the vector field $\bds A$ are independent of $\theta$, via Noether's Theorem.  We need two Hamiltonians -- $H(\bds x, \bds p)$ corresponding to particle dynamics in the field, and 
$$L_z(\bds x, \bds p) = xp_y-yp_x$$
Noether's Theorem states that $H$ is invariant under the $L_z$-flow if and only $L_z$ is invariant under the $H$ flow.  We check the former by calculating the $L_z$-flow
$$\dot {\bds x} = \nabla_{\bds p} L_z \,, \qquad \dot{\bds p} = -\nabla_{\bds x} L_z$$
This gives the equations
\begin{align*}
&\dot x = -y && \dot p_x = -p_y \\
&\dot y = x && \dot p_y = p_x
\end{align*}
For convenience, we assume that $(x_0,y_0)=(r,0)$.  Then $x(t)= r\cos t$, $y(t)=r\sin t$, from which we conclude that $r(t)=r$ and $\theta(t)=t$.  Also
$$\begin{bmatrix} p_x(t) \\ p_y(t) \end{bmatrix} = \begin{bmatrix} \cos t & -\sin t \\ \sin t & \cos t \end{bmatrix} \begin{bmatrix} (p_x)_0 \\ (p_y)_0 \end{bmatrix} = \begin{bmatrix} \cos \theta & -\sin \theta \\ \sin \theta & \cos \theta \end{bmatrix} \begin{bmatrix} (p_x)_0 \\ (p_y)_0 \end{bmatrix}$$
Since $(p_x,p_y)$ are related to $(p_r,p_\theta)$ by the same equation
$$\begin{bmatrix} p_x \\ p_y \end{bmatrix} = \begin{bmatrix} \cos \theta & -\sin \theta \\ \sin \theta & \cos \theta \end{bmatrix} \begin{bmatrix} p_r \\ p_\theta \end{bmatrix}$$
we obtain that $p_r(t)=(p_x)_0$ and $p_\theta(t) = (p_y)_0$ are both constant.  In summary, all three of $r(t)$, $p_r(t)$, and $p_\theta(t)$ are constant under this flow, while $\theta(t)=t$. (This can also be obtained from \eqref{E:Ham-cyl} below with $H$ replaced by $L_z$ although we must take care to use \eqref{E:dphi-p}).
Since
$$H = \frac{(p_r-qA_r)^2}{2m} + \frac{(p_\theta - qA_\theta)^2}{2m} + \frac{(p_z-qA_z)^2}{2m}+q\phi$$
is an expression that is independent of $\theta$, $H$ is constant under this flow.  
\end{example}

\begin{example}
We show that $\ell_z$ is conserved if $\phi$ and the cylindrical component $A_r$, $A_\theta$, and $A_z$ of the vector field $\bds A$  are independent of $\theta$, via direct calculation in cylindrical coordinates, using \eqref{E:dphi-p}.  Since $\bds x = r \bds e_r + z \bds e_z$, $\dot{\bds x} = \dot r \bds e_r + r\dot\theta \bds e_\theta + \dot z \bds e_z$.  Since $\bds p = p_r \bds e_r + p_\theta \bds e_\theta + p_z \bds p_z$, we have $\dot{\bds p} = (\dot p_r - \dot \theta p_\theta) \bds e_r + (\dot p_\theta + \dot \theta p_r) \bds e_\theta + \dot p_z \bds e_z $.  Via \eqref{E:grad-cyl},
$$\nabla_{\bds x} H = \partial_r H \, \bds e_r + \frac1r \partial_\theta H \, \bds e_\theta + \partial_z H \, \bds e_z$$
but for the $\bds p$ gradient, we simply project onto the orthonormal basis:
$$\nabla_{\bds p} H = \partial_{p_r} H \, \bds e_r + \partial_{p_\theta} H \, \bds e_\theta + \partial_{p_z} H \, \bds e_z $$
Therefore Hamilton's equations in cylindrical coordinates are:
\begin{equation}\label{E:Ham-cyl}
\begin{aligned}
&\dot r = \partial_{p_r} H && \dot p_r - \dot \theta p_\theta = - \partial_r H \\
&r\dot \theta = \partial_{p_\theta} H &\qquad \qquad & \dot p_\theta + \dot \theta p_r = - \frac1r \partial_\theta H \\
&\dot z = \partial_{p_z} H && \dot p_z = -\partial_zH
\end{aligned}
\end{equation}
Thus
\begin{equation}\label{E:lzdot}
\dot \ell_z = \dot r p_\theta + r \dot p_\theta = p_\theta \partial_{p_r} H - p_r \partial_{p_\theta} H - \partial_\theta H
\end{equation}
Now, even though $\phi$, $A_r$, $A_\theta$, and $A_z$ are all independent of $\theta$ and
$$H = \frac{(p_r-qA_r)^2}{2m} + \frac{(p_\theta -qA_\theta )^2}{2m} + \frac{(p_z-qA_z)^2}{2m} + q\phi  $$
we caution that $\partial_\theta H \neq 0$ due to the fact that $p_r$ and $p_\theta$ depend on $\theta$ -- see \eqref{E:dphi-p}.  We compute
$$\partial_{p_r} H = \frac{p_r-qA_r}{m} \,, \quad \partial_{p_\theta} H = \frac{p_\theta - q A_\theta}{m} \,, \qquad \partial_\theta H =  \frac{q}{m} ( p_rA_\theta - p_\theta A_r)$$
Plugging into \eqref{E:lzdot}, we obtain $0$.
\end{example}

Recall our assumptions that $\phi$ and the cylindrical components $A_r$, $A_\theta$, $A_z$ are independent of $z$ and $\theta$, which implies that $B_r=0$.   The equilibrium Maxwell equations take the form
\begin{equation}
\label{E:M-cyl}
\begin{aligned}
& - \frac1r \partial_r r \partial_r \phi = \epsilon_0^{-1} \rho &\qquad  \\
&  - \partial_r \frac1r \partial_r r   A_\theta = \mu_0 j_\theta \\
&- \frac1r \partial_r r \partial_r A_z = \mu_0 j_z 
\end{aligned}
\end{equation}
The Vlasov densities 
$$f^\pm = \mu^\pm( e^\pm, \ell_z^\pm, p_z^\pm)$$
Since $\ell_z^\pm$ and $p_z^\pm$ do not depend on $v_r$ and $e^\pm$ is even in $v_r$, it follows that $\bds j_r=0$, and hence we can take $A_r=0$.  From \eqref{E:curlA}, $B_\theta= - \partial_r A_z$ and $B_z = \frac1r\partial_r (rA_\theta)$.   

We now consider two special cases:  In \S\ref{S:z-pinch}, we consider the special case $B_z=0$ equivalently $A_\theta=0$, and in \S\ref{S:theta-pinch}, we consider the special case $B_\theta=0$ equivalently $A_z=0$.

\section{$z$-pinch type equilibria}
\label{S:z-pinch}

In this section, we consider equilibria with $B_z=0$ equivalently $A_\theta=0$, that we call \emph{$z$-pinch type}.  Since $A_\theta=0$, there is no current in the $\theta$ direction, and hence we might as well further assume that $\mu^\pm(e^\pm, p_z^\pm)$ with no $\ell_z^\pm$ dependence\footnote{We could entertain the case in which the two species have cancelling $\theta$-current, but we have decided not to pursue this here.}.   We show that \eqref{E:M-cyl} becomes a coupled system of second-order ODEs  for $\phi$ and $A_z$, with dependent variable $r$.

 Since $e$, $p_z$ depend only on $\alpha = \sqrt{v_r^2+v_\theta^2}$ and $v_z$ we first change to ``polar coordinates'':
\begin{equation}
\label{E:jz}
j_z(\phi,A_z) = \sum_{\pm} 2\pi q^\pm \int_{\alpha=0}^\infty \int_{v_z=-\infty}^{+\infty}  v_z \mu^\pm( e^\pm(\alpha, v_z, r), p_z^\pm (v_z,r)) \, \alpha d\alpha dv_z
\end{equation}

\begin{remark}
It might seem expedient to change integration variables from $(\alpha, v_z)$ to $(e, p_z)$.  Since $e = \frac12 m(\alpha^2+v_z^2) + q\phi(r)$ and $p_z = mv_z + qA_z(r)$, the Jacobian is
 $$\begin{vmatrix} \frac{\partial e}{\partial \alpha} & \frac{\partial e}{\partial v_z} \\ \frac{\partial p_z}{\partial \alpha} & \frac{\partial p_z}{\partial v_z} \end{vmatrix} = \begin{vmatrix} m\alpha & mv_z \\ 0 & m \end{vmatrix} = m^2\alpha$$
 Thus
 $$j_z(\theta, A_z) = \sum_{\pm} \frac{2\pi q^\pm}{(m^\pm)^3} \int_{e=q^\pm \phi(r)}^{+\infty} \int_{|p_z-a^\pm A_z(r)|\leq \sqrt{2m^\pm(e-q^\pm \phi(r))}} (p_z - q^\pm A_z(r)) \mu^\pm(e, p_z) \, dp_z\, de$$
 However, as we can see, the integration domain depends on $\phi(r)$ and $A_z(r)$ in a complicated way.  Therefore, we will prefer to use \eqref{E:jz} directly.
\end{remark}

Similar to \eqref{E:jz}, we find the following formula for the charge density $\rho$
\begin{equation}
\label{E:rho}
\rho(\phi, A_z) =  \sum_\pm 2\pi q^\pm \int_{\alpha=0}^\infty \int_{v_z=-\infty}^{+\infty}   \mu^\pm( e^\pm(\alpha, v_z, r), p_z^\pm (v_z,r)) \, \alpha d\alpha dv_z
\end{equation}

The equations that we need to solve are given in \eqref{E:M-cyl}.

\subsection{Constant drift $z$-pinch\label{SS:z-const}}

In this section, we prove Theorem \ref{T:z-pinch-const}, which provides an equilibrium profile \eqref{E:z-mu-pm} for which the corresponding field equation for $A_z$ reduces to \eqref{E:zp07}.

Suppose that
\begin{equation}\label{E:z-mu-pm}
\mu^\pm(e^\pm, p_z^\pm) = \frac{\omega^\pm}{\beta^2\zeta (\sigma^\pm)^3} \exp \left( \frac{- (\gamma^\pm)^2 e^\pm}{m^\pm (\sigma^\pm)^2} + \frac{ \gamma^\pm  p_z^\pm}{m^\pm \sigma^\pm} \right)
\end{equation}
where $\omega^\pm > 0$ and  $\gamma^\pm\in \mathbb{R}$ are dimensionless constants; $\gamma^\pm$  controls the $z$-velocity drift of $\gamma^\pm \sigma^\pm$.   We start by assuming $\phi=0$, i.e. electrical neutrality, which forces charge neutrality, $\rho=0$.   Plug in $e^\pm= \frac12m^\pm(\alpha^2 + v_z^2)$ (assuming $\phi=0$) and $p_z^\pm= m^\pm v_z + q^\pm A_z$, recalling $\alpha^2 = v_r^2 + v_\theta^2$.  For convenience, we write the expression in terms of the dimensionless quantities
$$\tilde \alpha^\pm  = \frac{\gamma^\pm \alpha}{\sigma^\pm} \,, \qquad \tilde v_z^\pm = \frac{\gamma^\pm v_z}{\sigma^\pm} \,, \qquad \tilde A_z^\pm = \frac{\gamma^\pm q^\pm A_z}{m^\pm \sigma^\pm}$$
(The $\pm$ superscripts on $\tilde v_z$, $\tilde \alpha$ are dropped below for readability.)  The result is that the exponent in \eqref{E:z-mu-pm} can be re-expressed as:
\begin{equation}
\label{E:zp05}
\begin{aligned}
\frac{- (\gamma^\pm)^2 e^\pm}{m^\pm (\sigma^\pm)^2} + \frac{ \gamma^\pm  p_z^\pm}{m^\pm \sigma^\pm} 
&= -\tfrac12 \tilde \alpha^2 - \tfrac12 \tilde v_z^2 +\tilde v_z + \tilde A_z^\pm \\
&= -\tfrac12 \tilde \alpha^2 -\tfrac12 (\tilde v_z - 1)^2 +\tfrac12  + \tilde A_z^\pm
\end{aligned}
\end{equation}
From \eqref{E:zp05}, \eqref{E:z-mu-pm}, and \eqref{E:rho}, we have
\begin{equation}\label{E:zp01}
\begin{aligned}
\rho(A_z) &= \sum_{\pm} \frac{2\pi q^\pm \omega^\pm}{\beta^2\zeta\, |\gamma^\pm|^3} e^{1/2} e^{\tilde A_z^\pm} \int_{\tilde \alpha=0}^{+\infty} e^{-\tilde\alpha^2/2} \tilde \alpha \, d\tilde \alpha \int_{\tilde v_z = -\infty}^{\infty} e^{-(\tilde v_z - 1)^2/2} \, d\tilde v_z\\
&= \frac{(2\pi)^{3/2} \,e^{1/2}}{\beta^2\zeta} \sum_{\pm} \frac{q^\pm \omega^\pm}{|\gamma^\pm|^3} \exp\left( \frac{\gamma^\pm q^\pm}{m^\pm \sigma^\pm} A_z \right)
\end{aligned}
\end{equation}
Also from \eqref{E:zp05}, \eqref{E:z-mu-pm}, and \eqref{E:jz}
\begin{equation}\label{E:zp04}
\begin{aligned}
j_z(A_z) &= \frac{ 2\pi \, e^{1/2}}{\beta^2 \zeta} \sum_{\pm} \frac{ q^\pm \, \omega^\pm \, \sigma^\pm }{|\gamma^\pm|^3\, \gamma^\pm} e^{\tilde A_z^\pm} \int_{\tilde \alpha=0}^{+\infty} e^{-\tilde\alpha^2/2} \tilde \alpha \, d\tilde \alpha \int_{\tilde v_z = -\infty}^{\infty} \tilde v_z e^{-(\tilde v_z - 1)^2/2} \, d\tilde v_z\\
&= \frac{ (2\pi)^{3/2} \, e^{1/2} }{ \beta^2 \, \zeta } \sum_{\pm} \frac{ q^\pm \,\omega^\pm \,\sigma^\pm}{|\gamma^\pm|^3\, \gamma^\pm } \exp\left( \frac{\gamma^\pm q^\pm}{m^\pm \sigma^\pm} A_z \right)
\end{aligned}
\end{equation}
By \eqref{E:zp01}, the condition $\rho(A_z)=0$ forces 
\begin{equation}\label{E:zp02}
\frac{\gamma^+ q^+}{m^+\sigma^+} = \frac{\gamma^- q^-}{m^-\sigma^-} \,, \qquad \frac{q^+\omega^+}{|\gamma^+|^3}  = - \frac{q^- \omega^-}{|\gamma^-|^3}
\end{equation}
Note that condition \eqref{E:zp02} implies that $\gamma^+$ and $\gamma^-$ have opposite sign.  In place of $\gamma^\pm$ and $\omega^\pm$ it is easier to express $j_z(A_z)$ in terms of
$$\lambda \defeq  \frac{q^+\gamma^+}{m^+\sigma^+} = \frac{q^-\gamma^-}{m^-\sigma^-}\,, \qquad a \defeq \frac{q^+\omega^+}{|\gamma^+|^3}  = - \frac{q^- \omega^-}{|\gamma^-|^3}$$
Now $-\infty<\lambda<\infty$ and $a>0$ can be thought of as two free parameters, and 
$$j_z(A_z) = \frac{ (2\pi)^{3/2} e^{1/2} a}{\beta^2 \, \zeta \, \lambda} \left( \sum_\pm \frac{|q^\pm|}{m^\pm} \right) e^{\lambda A_z}$$
 Let
$$\bar A_z = \lambda A_z \,, \qquad \bar r = \frac{r}{\beta}$$
so that the function $\bar A_z$ and the variable $\bar r$ are now dimensionless.  Also, define the dimensionless constant
$$d \defeq \frac{ (2\pi)^{3/2} e^{1/2} \, a \, \mu_0}{  \zeta} \left( \sum_\pm \frac{|q^\pm|}{m^\pm} \right) >0$$
Then \eqref{E:M-cyl} becomes
\begin{equation}\label{E:zp07}
\boxed{-\frac{1}{\bar r} \, \partial_{\bar r} \, \bar r \, \partial_{\bar r} \, \bar A_z = d \, e^{\bar A_z}}
\end{equation}
Remarkably, \eqref{E:zp07} has a Hamiltonian formulation. 
Change variable
$$\bar r = e^{ s}  \qquad \implies \qquad \frac{d}{ds} = \bar r \frac{d}{d\bar r}$$
Then $-\infty<s<\infty$ and \eqref{E:zp07} becomes
$$-\partial_s^2 \bar A_z = d  e^{\bar A_z + 2s}$$
One final adjustment, letting
$$\hat A_z = \bar A_z+ 2s + \ln d$$
reduces the problem to an autonomous ODE
$$-\partial_s^2 \hat A_z = e^{\hat A_z}$$
where $s$ and $\hat A_z$ are dimensionless.  This allows for easy analysis of the possible solutions of \eqref{E:zp07}.  However, as remarked in the introduction after the statement of Theorem \ref{T:z-pinch-const}, all solutions satisfying the physical boundary conditions at $\bar r=0$ and $\bar r \to +\infty$ are already given by the explicit formula \eqref{E:Bennett}.   Note that the statement of Theorem \ref{T:z-pinch-const} is obtained by reducing the above calculation to the case $\gamma^+=1$, $\gamma^-=-1$.

\subsection{Variable drift $z$-pinch\label{SS:var-z}}
In this section, we prove Theorem \ref{T:z-pinch-var}, which provides an equilibrium profile \eqref{E:modz-mu-pm} for which the corresponding field equation for $A_z$ is shown to reduce to \eqref{E:zp09} below.

We now consider
\begin{equation}\label{E:modz-mu-pm}
\mu^\pm(e^\pm, p_z^\pm) =  \frac{\omega^\pm}{\beta^2\zeta (\sigma^\pm)^3} \exp \left( -\frac{(\gamma^\pm)^2 e^\pm}{m^\pm (\sigma^\pm)^2}  - \frac{(\gamma^\pm)^2 (p_z^\pm)^2}{2(m^\pm)^2 (\sigma^\pm)^2}\right)
\end{equation}

\begin{remark}  We could consider a more general quadratic expression in the exponential:
$$
\mu^\pm(e^\pm, p_z^\pm) =  \frac{\omega^\pm}{\beta^2\zeta (\sigma^\pm)^3} \exp \left( -\frac{(\gamma^\pm)^2 e^\pm}{m^\pm (\sigma^\pm)^2} + B \frac{\gamma^\pm p_z^\pm}{m^\pm \sigma^\pm} - D\frac{(\gamma^\pm)^2 (p_z^\pm)^2}{2(m^\pm)^2 (\sigma^\pm)^2}\right)
$$
for $B\in \mathbb{R}$ and $D>0$ dimensionless constants, independent of the $\pm$ species.  However, we can absorb the linear term (with $B$ coefficient) as a shift of $A_z$, which has no effect on $\bds B$.   Then by replacing $\sqrt{D/(1+D)}A_z$ by $A_z$, we can effectively eliminate the this coefficient, burying its effect into the constant $d>0$ defined below.  So for simplicity, we take $B=0$ and $D=1$.
\end{remark}

Plug in $e^\pm= \frac12m^\pm(\alpha^2 + v_z^2)$ (assuming $\phi=0$) and $p_z^\pm= m^\pm v_z + q^\pm A_z$, recalling $\alpha^2 = v_r^2 + v_\theta^2$.  For convenience, we write the result in terms of the dimensionless quantities
$$\tilde \alpha^\pm  = \frac{\gamma^\pm \alpha}{\sigma^\pm} \,, \qquad \tilde v_z^\pm = \frac{\gamma^\pm v_z}{\sigma^\pm} \,, \qquad \tilde A_z^\pm = \frac{\gamma^\pm q^\pm A_z}{m^\pm \sigma^\pm}$$
The result is that the exponent in \eqref{E:modz-mu-pm} can be re-expressed upon completing the square in $v_z$ as
\begin{equation}
\label{E:cts}
 -\frac{(\gamma^\pm)^2 e^\pm}{m^\pm (\sigma^\pm)^2}  - \frac{(\gamma^\pm)^2 (p_z^\pm)^2}{2(m^\pm)^2 (\sigma^\pm)^2} = -\tfrac12 (\tilde \alpha^\pm)^2 - ( \tilde v_z^\pm + \tfrac12\tilde A_z^\pm )^2 - \tfrac14(\tilde A_z^\pm )^2 
\end{equation}
The $r$-dependence comes entirely through the $A_z(r)$ terms.  By \eqref{E:rho},
$$
 \rho(A_z) = \sum_{\pm} \frac{2\pi q^\pm \omega^\pm}{\beta^2\zeta \, |\gamma^\pm|^3} \int_{\tilde \alpha^\pm=0}^{+\infty} \int_{\tilde v_z^\pm =-\infty}^{+\infty} \exp(\cdots) \, \tilde \alpha^\pm d\tilde\alpha^\pm d\tilde v_z^\pm
$$
where $(\cdots)$ is replaced by the terms on the right side of \eqref{E:cts}.  The result after carrying out the integrals is:
\begin{equation}
\label{E:rho-zp2}
\begin{aligned}
\rho(A_z) &=  \frac{ (2\pi)^{3/2}}{\beta^2 \zeta} \sum_\pm \frac{q^\pm \omega^\pm}{|\gamma^\pm|^3} \exp \left[ -\frac14 \left( \frac{\gamma^\pm q^\pm A_z}{m^\pm \sigma^\pm} \right)^2 \right]
\end{aligned}
\end{equation}
By \eqref{E:M-cyl}, to achieve $\phi=0$, we need $\rho(A_z)=0$ for all $A_z$, so there must be $\pm$ independent constants $\lambda$ and $a$ so that
\begin{equation}
\lambda \defeq \frac{\gamma^\pm q^\pm}{m^\pm \sigma^\pm} \in \mathbb{R} \,, \qquad a = \frac{q^+ \omega^+}{|\gamma^+|^3} = -\frac{q^- \omega^+}{|\gamma^+|^3} >0
\end{equation}
(recall that $q+>0$ and $q^-<0$). By \eqref{E:jz},
$$
 j_z(A_z) = \sum_{\pm} \frac{2\pi q^\pm \omega^\pm \sigma^\pm }{\beta^2\zeta \, |\gamma^\pm|^3 \gamma^\pm } \int_{\tilde \alpha^\pm=0}^{+\infty} \int_{\tilde v_z^\pm =-\infty}^{+\infty} \tilde v_z \exp(\cdots) \, \tilde \alpha^\pm d\tilde\alpha^\pm d\tilde v_z^\pm
$$
Utilizing \eqref{E:cts} for the quantity in parentheses,
$$j_z(A_z) =   -\frac{ (2\pi)^{3/2} }{2 \beta^2 \zeta} \sum_\pm \frac{q^\pm \omega^\pm \sigma^\pm }{|\gamma^\pm|^3 \gamma^\pm } \left( \frac{\gamma^\pm q^\pm A_z}{m^\pm \sigma^\pm}  \right) \exp \left[ -\frac14 \left( \frac{\gamma^\pm q^\pm A_z}{m^\pm \sigma^\pm} \right)^2 \right]$$
$$ = -\frac{ (2\pi)^{3/2}  a}{2\beta^2 \zeta \lambda} \left( \sum_\pm \frac{q^\pm}{m^\pm} \right) \lambda A_z  \exp[ -\tfrac14\lambda^2 A^2]$$
By \eqref{E:M-cyl}, we need to consider the solutions to
\begin{equation}
\label{E:zp08}
-\partial_r^2 A_z - \frac1r \partial_r A_z = \mu_0 j_z(A_z)
\end{equation}
Let
$$\bar A_z \defeq \frac{\lambda A_z }{\sqrt 2} \,, \qquad \bar r \defeq \frac{r}{\beta}\,, \qquad d \defeq \frac{ (2\pi)^{3/2} \, a \, \mu_0}{2 \zeta } \left( \sum_\pm \frac{q^\pm}{m^\pm} \right) $$
so that $\bar A_z$ a dimensionless function, $\bar r$ is a dimensionless variable, and $d>0$ is a dimensionless constant determined from physical parameters as well as the structural input parameters $a$ and $\lambda$.  Then \eqref{E:zp08} becomes

\begin{equation}
\label{E:zp09}
-\partial_{\bar r}^2 \, \bar A_z - \frac1{\bar r} \, \partial_{\bar r} \, \bar A_z = - d \,\bar A_z \,e^{-\frac12\bar A_z^2}
\end{equation}
To obtain the statement of Theorem \ref{T:z-pinch-var}, take $\gamma^\pm =\pm 1$.

We can in fact dispose of the constant $d$ by replacing $r$ by $\sqrt d \, \bar r$, so in the following we shall take $d=1$ and study the equation
\begin{equation}
\label{E:mod-zp-model}
-\partial_r^2A - \frac1r \partial_r A = -A e^{-A^2/2}
\end{equation}

At first sight, we might hope for a solution that decays exponentially $A(r) \to 0$ as $r\to +\infty$, following the linear equation
$$-\partial_r^2 A - \frac1r \partial_r A = -A$$
which has as a solution the modified Bessel function of the second kind $K_0(r)$.   Indeed, such solutions exist, but at question is their behavior as $r \to 0^+$.  The mechanical analysis below shows that any solution to \eqref{E:mod-zp-model} that follows the exponentially decaying form of $K_0(r)$ as $r\to +\infty$ must in fact blow-up as $r\to 0^+$ (It cannot have $\partial_r A(r) \to 0$ as $r\to 0^+$, so cannot be smooth at $r=0$.)


Since we demand solutions that are regular at $r=0$, i.e. for which $\partial_r A(0)=0$ and $A(0)\neq 0$, we are forced to have $|A(r)| \to \infty$ as $r\to \infty$.  In this case, as $r\to +\infty$, \eqref{E:mod-zp-model} becomes the homogeneous equation
$$-\partial_r^2 A - \frac1r \partial_r A =0$$
which has solution $A(r) = d \ln r$ for some constant $d \neq 0$ (positive or negative).   Solutions with $A(r) \sim \ln r$ as $r\to +\infty$ are not completely unconfined since $B_\theta = - \partial_r A_z \sim \frac{1}{r}$, so that magnetic field just barely fails to have finite energy.

We present a mechanical analysis showing the impossibility of exponentially decaying solutions.
Consider \eqref{E:mod-zp-model}, i.e.
\begin{equation}
\label{E:modelODE}
-A'' - \frac1r A' = -A e^{-A^2/2}
\end{equation}
where ${}'$ denotes the $r$ derivative. Taking the ``mechanical energy''
\begin{equation}
\label{E:mech-en}
H[A,A'] \defeq \frac12 (A')^2 + e^{-A^2/2} \,,
\end{equation}
and use the shorthand $H(r) = H[A(r),A'(r)]$. Applying \eqref{E:modelODE}, we obtain
\begin{equation}
\label{E:mech-en-dis}
\partial_r H(r) = - \frac{1}{r} (A'(r))^2 <0
\end{equation}

Thus solutions lose energy as we evolve from ``time'' $r=0$ to $r=+\infty$.  We seek a solution satisfying
$$ \text{at }r=0 \,, \qquad A'(0)=0$$
$$\text{as }r\to +\infty, \qquad A(r) \to 0 \,, \quad A'(r) \to 0$$
Thus, as $r\to +\infty$, we will have $H(r) \to 1$, and by \eqref{E:mech-en-dis}, we must have $H(0)>1$.  But $A'(0)=0$, so \eqref{E:mech-en} gives a contradiction.

\section{$\theta$-pinch type equilibria}
\label{S:theta-pinch}

In this section, we consider equilibria with $B_\theta=0$ equivalently $A_z=0$, that we call \emph{$\theta$-pinch type}.  
This case gives a current in the $\theta$ direction.  Since $A_z=0$, there is no current in the $z$-direction, and hence we might as well further assume that $\mu^\pm(e^\pm, \ell_z^\pm)$ with no $p_z^\pm$ dependence.   We show that \eqref{E:M-cyl} becomes a coupled system of second-order ODEs for $\phi$ and $r A_\theta$, with dependent variable $r$.

Since $e$, $\ell_z$ depend only on $\alpha = \sqrt{v_r^2+v_z^2}$ and $v_\theta$ we first change to ``polar coordinates'':
\begin{equation}
\label{E:jphi}
j_\theta(\phi,rA_\theta) = \sum_{\pm} 2\pi q^\pm \int_{\alpha=0}^\infty \int_{v_\theta=-\infty}^{+\infty}  v_\theta \mu^\pm( e^\pm(\alpha, v_\theta, r), \ell_z^\pm (v_\theta,r)) \, \alpha d\alpha dv_\theta
\end{equation}

Similar to \eqref{E:jphi}, we find the following formula for the charge density $\rho$
\begin{equation}
\label{E:tp-rho}
\rho(\phi, rA_\theta) =  \sum_\pm 2\pi q^\pm \int_{\alpha=0}^\infty \int_{v_\theta=-\infty}^{+\infty}   \mu^\pm( e^\pm(\alpha, v_\theta, r), \ell_z^\pm (v_\theta,r)) \, \alpha d\alpha dv_\theta
\end{equation}

The equations that we need to solve are given in \eqref{E:M-cyl}.

\subsection{Variable drift $\theta$-pinch\label{SS:t-var}}

In this section, we prove Theorem \ref{T:theta-pinch-var} which provides an equilibrium profile \eqref{E:modt-mu-pm} and the scaled field equation is shown to satisfy \eqref{E:tp-scaled-field}.

Consider
\begin{equation}\label{E:modt-mu-pm}
\mu^\pm(e^\pm, \ell_z^\pm) =  \frac{\omega^\pm}{\beta^2\zeta (\sigma^\pm)^3} \exp \left( -\frac{(\gamma^\pm)^2 e^\pm}{m^\pm (\sigma^\pm)^2}  - \frac{(\gamma^\pm)^2 (\ell_z^\pm)^2}{2(m^\pm)^2 (\sigma^\pm)^2 \beta^2}\right)
\end{equation}
Here, $\gamma^\pm$, $\omega^\pm$ are dimensionless constants.

Plug in $e^\pm= \frac12m^\pm(\alpha^2 + v_\theta^2)$ (assuming $\phi=0$) and $\ell_z^\pm= r(m^\pm v_\theta + q^\pm A_\theta)$, recalling $\alpha^2 = v_r^2 + v_z^2$.  For convenience, we write the result in terms of the dimensionless quantities
$$\tilde \alpha^\pm  = \frac{\gamma^\pm \alpha}{\sigma^\pm} \,, \qquad \tilde v_\theta^\pm = \frac{\gamma^\pm v_\theta}{\sigma^\pm} \,, \qquad \tilde A_\theta^\pm = \frac{\gamma^\pm q^\pm A_\theta}{m^\pm \sigma^\pm} \,, \qquad \bar r = \frac{r}{\beta}$$
The result is that the exponent in \eqref{E:modt-mu-pm} can be re-expressed upon completing the square in $v_\theta$ as
\begin{equation}
\label{E:tp-exp}
\begin{aligned} 
\indentalign -\frac{(\gamma^\pm)^2e^\pm}{m^\pm (\sigma^\pm)^2}  - \frac{(\gamma^\pm)^2 (\ell_z^\pm)^2}{2(m^\pm)^2 (\sigma^\pm)^2 \beta^2} \\
&= - \tfrac12 (\tilde \alpha^\pm )^2 - \tfrac12 (1 + \bar r^2) \left( \tilde v_\theta + \frac{\bar r^2 \tilde A_\theta}{1+\bar r^2} \right)^2 - \frac{(\bar r\tilde A_\theta )^2}{2(1+\bar r^2)} 
\end{aligned}
\end{equation}
From \eqref{E:tp-rho}, \eqref{E:modt-mu-pm}, and \eqref{E:tp-exp},
$$\rho(A_\theta) =  \sum_\pm \frac{(2\pi)^{3/2} q^\pm \omega^\pm}{\beta^2\zeta |\gamma^\pm|^3} \frac{1}{(1+\bar r^2)^{1/2}} \exp \left[ - \frac{ ( \bar r \tilde A_\theta^\pm )^2}{2(1+\bar r^2)} \right] $$

The condition $\rho(A_\theta)=0$ forces 
$$\lambda \defeq \frac{\gamma^+ q^+}{m^+\sigma^+} = \frac{\gamma^- q^-}{m^- \sigma^-} \,, \qquad
 a \defeq \frac{q^+ \omega^+}{|\gamma^+|^3} = \frac{-q^-\omega^-}{|\gamma^-|^3}>0$$
 Originally, there were four parameters $\gamma^\pm$, $\omega^\pm$, and these two equations reduce the problem to two degrees of freedom.  Note that $\gamma^+$ and $\gamma^-$ have opposite sign since $q^+$ and $q^-$ have opposite sign.  We prefer, in the subsequent analysis, to regard $\lambda$ and $a>0$ as given free parameters.  Then we can recover the four original parameters $\gamma^\pm$, $\omega^\pm$ as
$$\gamma^\pm = \lambda \frac{m^\pm\sigma^\pm}{q^\pm} \,, \qquad \omega^\pm = a|\lambda|^3 \frac{(m^\pm)^3(\sigma^\pm)^3}{(q^\pm)^4}$$
From \eqref{E:jphi}, \eqref{E:modt-mu-pm}, and \eqref{E:tp-exp},
\begin{align*}
j_\theta(A_\theta) &= - \sum_\pm \frac{ (2\pi)^{3/2} \, q^\pm  \omega^\pm \sigma^\pm }{  \, \beta^2 \zeta \, |\gamma^\pm|^3 \gamma^\pm} \exp\left[ - \frac{ (\bar r  \lambda A_\theta)^2}{2(1 + \bar r^2)} \right] \frac{ \bar r^2  \lambda A_\theta }{(1 + \bar r^2)^{3/2}} \\
& = - \frac{ (2\pi)^{3/2}  \, a}{\beta^2 \zeta \lambda} \left(\sum_\pm \frac{|q^\pm|}{m^\pm } \right) \exp\left[ - \frac{ ( \bar r \lambda A_\theta )^2}{2(1 + \bar r^2)} \right] \frac{ \bar r^2  \lambda A_\theta}{(1 + \bar r^2)^{3/2}}
\end{align*}
The corresponding Maxwell equation is
$$ -\partial_r \, \frac{1}{r} \, \partial_r \, r \,  A_\theta = \mu_o j_\theta$$
which leads to
$$ -\partial_{\bar r} \, \frac{1}{\bar r} \, \partial_{\bar r} \, ( \bar r \lambda  A_\theta) \\
= - \frac{ (2\pi)^{3/2}  \, a \, \mu_0}{\zeta } \left(\sum_\pm \frac{|q^\pm|}{m^\pm } \right) \exp\left[ - \frac{ ( \bar r \lambda A_\theta )^2}{2(1 + \bar r^2)} \right] \frac{ \bar r^2 \lambda A_\theta}{(1 + \bar r^2)^{3/2}}
$$
Define $\bar A_\theta$ by
$$\bar A_\theta = \lambda A_\theta  \,, \qquad d = \frac{ (2\pi)^{3/2} \, a \, \mu_0}{ \zeta } \sum_\pm \frac{|q^\pm|}{m^\pm } $$
The shift by a constant will not affect the magnetic field since $B_z=  \frac{1}{r} \partial_r (r A_\theta)$.  With this transformation, the equation takes the form
\begin{equation}
\label{E:tp-scaled-field}
\Big(- \partial_{\bar r}^2 - \frac{1}{\bar r} \partial_{\bar r} + \frac{1}{\bar r^2} \Big)\bar A_\theta = -\partial_{\bar r} \, \frac{1}{\bar r} \, \partial_{\bar r} \,  \bar r A_\theta = - d \frac{ \bar r^2 \bar A_\theta}{(1 + \bar r^2)^{3/2}}\exp\left[ - \frac{ \bar r^2 \bar A_\theta^2}{2(1 + \bar r^2)} \right] 
\end{equation}
We obtain Theorem \ref{T:theta-pinch-var} by taking $\gamma^\pm=\pm 1$.
Let us investigate the asymptotics of solutions to \eqref{E:tp-scaled-field} as $\bar r \to +\infty$.  
At this point, we drop the bars in the notation. We consider the $r\to \infty$, $A(r) \to 0$ limit, so we drop the exponential term.  Let
$$s = \sqrt{2r} \,, \qquad \frac{d}{dr} = \frac{1}{s} \frac{d}{ds}$$
Then the above equation converts to 
$$ \Big(- \frac{d^2}{ds^2} - \frac{1}{s} \frac{d}{ds} + 2d + 4s^{-2}\Big) A =0$$
This has exponentially decaying solution 
$$A(s) = K_2( \sqrt{2d}\, s) \sim \frac{\pi}{2} (\sqrt{2d}\, s)^{-1/2} e^{-\sqrt{2d} \,s}$$
However, no solution of \eqref{E:tp-scaled-field} with this exponential decay as $r\to +\infty$ can have a finite limit as $r\searrow 0$.  As remarked in the introduction, the only physical solutions to \eqref{E:tp-scaled-field} satisfy $\bar A_\theta(\bar r) \sim r$ as $\bar r \to +\infty$.  

\section{Screw pinch type equilibria\label{S:screw}}

Recall that $0\leq z < \zeta$ with periodic boundary conditions connecting $z=0$ and $z=\zeta$.  Employ cylindrical coordinates and fix a pitch parameter $k\in \zeta \mathbb{Z}/2\pi$.  Consider helical translation parameterized by $\eta \in \mathbb{R}$, 
$$(r_0, \theta_0, z_0) \mapsto (r_0, \theta_0+ \eta \mod 2\pi, z_0+k\eta \mod \zeta)$$
We would like to consider this as ``flow in the $\eta$-direction'' so we work with spatial coordinates $(r,\eta, \sigma)$ where
\begin{equation}\label{E:sp01}
\eta = \theta \mod 2\pi \,, \qquad \sigma = z - k \theta \mod \zeta
\end{equation}
so the case $k=0$ reduces to $\eta=\theta$ and $\sigma=z$. 
With these coordinates,
$$\partial_\theta = \partial_\eta - k \partial_\sigma \,, \qquad \partial_z = \partial_\sigma$$
equivalently,
$$\partial_\eta = \partial_\theta + k \partial_z \,, \qquad \partial_\sigma =  \partial_z$$
 Now $\partial_\eta$ is the Killing field generating the symmetry of \emph{independence of $\eta$}, i.e.
\begin{enumerate}[left=0pt,label=$\bullet$]
\item $f$ depends only on $(r,\sigma)$
\item The cylindrical components $E_r$, $E_\theta$, $E_z$ of $\bds E$, and the cylindrical components $B_r$, $B_\theta$, and $B_z$  of $\bds B$, depend only on $(r,\sigma)$.
\end{enumerate}

The particle trajectory invariant is $\ell_z + kp_z = r(mv_\theta+qA_\theta) + k(mv_z + qA_z)$, and this suggests the following orthonormal frame on \emph{tangent spaces}:
\begin{equation}
\label{E:sp02}
\bds e_\eta = \frac{r \bds e_\theta + k \bds e_z}{\sqrt{r^2+k^2}} \,, \qquad \bds e_\sigma = \frac{-k\bds e_\theta + r \bds e_z}{\sqrt{r^2+k^2}}
\end{equation}
(Note that $\partial_\eta$ is parallel to $\bds e_\eta$ but $\partial_\sigma$ is not parallel to $\bds e_\sigma$.)  In this frame, writing $v_\theta \bds e_\theta + v_z \bds e_z = v_\eta \bds e_\eta + v_\sigma \bds e_\sigma$ yields the coordinate relations
\begin{equation}
\label{E:sp03}
v_\theta = \frac{rv_\eta - k v_\sigma}{\sqrt{r^2+k^2}} \,, \qquad v_z = \frac{kv_\eta + r v_\sigma}{\sqrt{r^2+k^2}}
\end{equation}
equivalently,
$$v_\eta = \frac{rv_\theta+kv_z}{\sqrt{r^2+k^2}} \,, \qquad v_\sigma = \frac{-kv_\theta + r v_z}{\sqrt{r^2+k^2}}$$
Similarly, writing $A_\theta \bds e_\theta + A_z \bds e_z = A_\eta \bds e_\eta + A_\sigma \bds e_\sigma$ yields the coordinate relations
\begin{equation}
\label{E:sp04}
A_\theta = \frac{rA_\eta - k A_\sigma}{\sqrt{r^2+k^2}} \,, \qquad A_z = \frac{kA_\eta + r A_\sigma}{\sqrt{r^2+k^2}}
\end{equation}
equivalently,
$$A_\eta = \frac{rA_\theta+kA_z}{\sqrt{r^2+k^2}} \,, \qquad A_\sigma = \frac{-kA_\theta + r A_z}{\sqrt{r^2+k^2}}$$
The velocity decomposition \eqref{E:sp03} induces a current density decomposition with components $j_r$, $j_\eta$, $j_\sigma$ with, in particular
$$j_\eta  = \iiint_{v_r,v_\eta,v_\sigma} v_\eta f^\pm \, dv_r \, dv_\eta \, dv_\sigma \,, \quad j_\sigma  = \iiint_{v_r,v_\eta,v_\sigma} v_\sigma f^\pm \, dv_r \, dv_\eta \, dv_\sigma$$
The conversions \eqref{E:sp03} give
$$j_\theta = \frac{r}{\sqrt{r^2+k^2}} j_\eta - \frac{k}{\sqrt{r^2+k^2}} j_\sigma \,, \quad j_z = \frac{k}{\sqrt{r^2+k^2}} j_\eta + \frac{r}{\sqrt{r^2+k^2}} j_\sigma$$
In the equilibrium example below, $j_\sigma=0$, so in this case
\begin{equation}
\label{E:sp06}
j_\theta = \frac{r}{\sqrt{r^2+k^2}} \, j_\eta \,, \quad j_z = \frac{k}{\sqrt{r^2+k^2}} \, j_\eta \quad \implies \quad \boxed{kj_\theta = r j_z}
\end{equation}
Plugging \eqref{E:sp04}, \eqref{E:sp06} into the Amp\`ere equation \eqref{E:M-cyl},
\begin{equation}
\label{E:sp05}
\begin{aligned}
&  - \partial_r \, \frac1r \, \partial_r \, r\, \frac{rA_\eta - k A_\sigma}{\sqrt{r^2+k^2}}  = \mu_0 \frac{r}{\sqrt{r^2+k^2}} \, j_\eta\\
&- \frac1r \, \partial_r \, r \, \partial_r  \, \frac{kA_\eta + r A_\sigma}{\sqrt{r^2+k^2}} = \mu_0 \frac{k}{\sqrt{r^2+k^2}} \, j_\eta
\end{aligned}
\end{equation}
Multiply the first equation by $k$ and the second by $r$.  Then the two left-sides provide the relation:
$$k \,  \frac1r \, \partial_r \, r\, \frac{rA_\eta - k A_\sigma}{\sqrt{r^2+k^2}} =   \, r \, \partial_r  \, \frac{kA_\eta + r A_\sigma}{\sqrt{r^2+k^2}}$$
Using the commutator identity $\partial_r \, r^2 = r^2 \, \partial_r + 2r$ on the left side, 
$$k \, (r\partial_r + 2) \frac{A_\eta}{\sqrt{k^2+r^2}} - k^2 \, \frac{1}{r} \, \partial_r \, r \frac{A_\sigma}{\sqrt{k^2+r^2}} = k\, r \, \partial_r \, \frac{A_\eta}{\sqrt{k^2+r^2}} +  \, r \, \partial_r \, r \, \frac{A_\sigma}{\sqrt{k^2+r^2}}$$
Cancelling the two $k r \partial_r  A_\eta/\sqrt{k^2+r^2}$ terms, we obtain the key field component relationship
\begin{equation}
\label{E:AsigmaAeta}
\partial_r\frac{r \,A_\sigma}{\sqrt{k^2+r^2}} = \frac{2kr A_\eta}{(k^2+r^2)^{3/2}} \qquad (\text{when }j_\sigma=0)
\end{equation}
Reinserting this into \eqref{E:sp05}, both the first and second equations of \eqref{E:sp05} become the following reduced Amp\`ere equation in $A_\eta$:
\begin{equation}\label{E:sp08}
-\partial_r \, r \, \partial_r \, \frac{A_\eta}{\sqrt{k^2+r^2}} - 2\partial_r \frac{r^2 A_\eta}{(r^2+k^2)^{3/2}} = \mu_0 \frac{r}{\sqrt{r^2+k^2}} \, j_\eta
\end{equation}
Now we compute the induced relationship between $B_\theta$ and $B_z$ when $A_r=0$, $A_\theta$ and $A_z$ depend only on $r$,  under the assumption that $j_\sigma=0$ and hence  \eqref{E:AsigmaAeta} holds.  Inserting \eqref{E:sp04}, and using $\delta = \sqrt{r^2+k^2}$,
$$B_z = \frac1r \partial_r r A_\theta = \frac1r \partial_r \frac{r^2 A_\eta}{\delta} - \frac{k}{r} \partial_r \frac{rA_\sigma}{\delta} \,, \qquad B_\theta = - \partial_r A_z = -k\partial_r \frac{A_\eta}{\delta} - \partial_r \frac{rA_\sigma}{\delta}$$
Plugging in \eqref{E:AsigmaAeta},
$$B_z = \frac1r\partial_r \frac{ r^2A_\eta}{\delta} - \frac{2k^2A_\eta}{\delta^3}= r \partial_r \frac{A_\eta}{\delta} + 2 \frac{r^2 A_\eta}{\delta^3} \,, \qquad B_\theta = -k \partial_r \frac{A_\eta}{\delta} - \frac{2kr A_\eta}{\delta^3}$$
from which it follows that 
$$kB_z = -rB_\theta$$
Thus $\bds j$ is parallel to $\bds e_\eta$ and $\bds B$ is parallel to $\bds e_\sigma$.

\subsection{Variable drift screw pinch equilibrium\label{SS:s-var}}

In this subsection, we provide an equilibrium profile \eqref{E:modsp-mu-pm}. 

Consider
\begin{equation}\label{E:modsp-mu-pm}
\mu^\pm(e^\pm, \ell_z^\pm+kp_z^\pm) =  \frac{\omega^\pm}{\beta^2\zeta (\sigma^\pm)^3} \exp \left( -\frac{(\gamma^\pm)^2 e^\pm}{m^\pm (\sigma^\pm)^2}  - \frac{(\gamma^\pm)^2 (\ell_z^\pm+kp_z^\pm)^2}{2(m^\pm)^2 (\sigma^\pm)^2 \beta^2}\right)
\end{equation}
Here, $\gamma^\pm$, $\omega^\pm$ are dimensionless constants.

Plug in $e^\pm= \frac12m^\pm(v_r^2 + v_\theta^2+v_z^2)$ (assuming $\phi=0$) and $$\ell_z^\pm+kp_z^\pm = m^\pm (rv_\theta + kv_z) + q^\pm (rA_\theta+kA_z)\,.$$  
For convenience, introduce the dimensionless quantities
$$\tilde v_r^\pm = \frac{\gamma^\pm v_r}{\sigma^\pm} \,, \quad 
\tilde v_\theta^\pm = \frac{\gamma^\pm v_\theta}{\sigma^\pm} \,, \quad 
\tilde v_z^\pm = \frac{\gamma^\pm v_z}{\sigma^\pm} \,, \quad
\tilde A_\theta^\pm = \frac{\gamma^\pm q^\pm A_\theta}{m^\pm \sigma^\pm} \,, \quad 
\tilde A_z^\pm = \frac{\gamma^\pm q^\pm A_z}{m^\pm \sigma^\pm} $$
as well as
$$\bar r = \frac{r}{\beta} \,, \qquad \bar k = \frac{k}{\beta} \in \frac{ \zeta \, \mathbb{Z}}{\beta}\,, \qquad \delta(\bar r ) = \sqrt{\bar r^2+\bar k^2}$$
The conversion to $\bds e_\eta$, $\bds e_\sigma$ frame components is
$$ \tilde v_\eta^\pm = \frac{\bar r \tilde v_\theta^\pm + \bar k \tilde v_z^\pm}{\delta(\bar r)} \,, \quad
\tilde v_\sigma^\pm = \frac{ -\bar k \tilde v_\theta^\pm + \bar r \tilde v_z^\pm}{\delta(\bar r)} \,, \quad \tilde A_\eta^\pm = \frac{\bar r \tilde A_\theta^\pm + \bar k \tilde A_z^\pm}{\delta(\bar r)} \,, \quad
\tilde A_\sigma^\pm = \frac{ -\bar k \tilde A_\theta^\pm + \bar r \tilde A_z^\pm}{\delta(\bar r)}$$
The result is that the exponent in \eqref{E:modsp-mu-pm} can be re-expressed upon completing the square in $v_\theta$ as
\begin{equation}
\label{E:sp-exp}
\begin{aligned} 
\indentalign -\frac{(\gamma^\pm)^2e^\pm}{m^\pm (\sigma^\pm)^2}  - \frac{(\gamma^\pm)^2 (\ell_z^\pm+kp_z^\pm)^2}{2(m^\pm)^2 (\sigma^\pm)^2 \beta^2} \\
&= - \tfrac12 (\tilde v_r^\pm)^2 - \tfrac12  (\tilde v_\sigma^\pm)^2 -\tfrac12(1+\delta^2) \Big( \tilde v_\eta^\pm + \frac{\delta^2 \tilde A_\eta^\pm}{1+\delta^2}\Big)^2 - \frac{(\delta \tilde A_\eta^\pm )^2}{2(1+\delta^2)}
\end{aligned}
\end{equation}
From \eqref{E:modsp-mu-pm}, \eqref{E:sp-exp}, we obtain
$$\rho(A_\eta) = \frac{(2\pi)^{3/2}}{\beta^2 \zeta \sqrt{1+\delta^2}} \sum_\pm   \,  \frac{q^\pm \omega^\pm}{|\gamma^\pm|^3 } \, \exp\left[ - \frac{(\delta  \tilde A_\eta^\pm )^2}{2(1+\delta^2)}  \right]$$
In order to have $\rho(A_\eta)(r)=0$ for each $r$, we need $\pm$ independent constants $\lambda$ and $a$ defined as follows:
\begin{equation}\label{E:sp07}
\lambda = \frac{\gamma^\pm q^\pm}{m^\pm \sigma^\pm} \in \mathbb{R} \,, \qquad a = \frac{|q^\pm| \omega^\pm}{|\gamma^\pm|^3 }>0
\end{equation}
Alternatively, given constants $\lambda \in \mathbb{R}$ and $a>0$, we can \emph{define} the dimensionless constants
$$\gamma^\pm = \frac{\lambda m^\pm \sigma^\pm}{q^\pm}\in \mathbb{R} \,, \qquad \omega^\pm = \frac{|\lambda|^3 \, a \, (m^\pm)^3 (\sigma^\pm)^3}{(q^\pm)^4}>0$$
to be inserted in the form of \eqref{E:modsp-mu-pm}.  Now from \eqref{E:modsp-mu-pm}, \eqref{E:sp-exp} again, we obtain
$$j_\eta(A_\eta) = - \frac{(2\pi)^{3/2} }{\beta^2 \zeta  \, \sqrt{1+\delta^2}} \sum_\pm   \,  \frac{q^\pm \sigma^\pm \omega^\pm}{|\gamma^\pm|^3 \gamma^\pm } \, \frac{ \delta^2 \tilde A_\eta^\pm}{1+\delta^2} \exp\left[ - \frac{(\delta \tilde A_\eta^\pm)^2}{2(1+\delta^2)}  \right]$$
Using the notation of \eqref{E:sp07},
$$j_\eta(A_\eta) = - \left(\sum_\pm   \,  \frac{ |q^\pm|}{ m^\pm} \right) \frac{(2\pi)^{3/2} \, a }{\beta^2 \, \zeta \, \lambda} \, \frac{ \delta^2 \lambda A_\eta }{(1+\delta^2)^{3/2}} \exp\left[ - \frac{(\delta \lambda A_\eta)^2}{2(1+\delta^2)} \right]$$
Introduce $\bar A_\eta$ as follows:
$$  \lambda A_\eta  =  \bar A_\eta$$
Plugging this into \eqref{E:sp08},
\begin{equation}\label{E:sp09}
-\partial_{\bar r} \, \bar r \, \partial_{\bar r} \, \delta^{-1} \bar A_\eta - 2\partial_{\bar r} \, \delta^{-3} \bar r^2 \,\bar A_\eta = - d \, \frac{ \bar r \delta \bar A_\eta}{ (1+\delta^2)^{3/2}} \exp\left[ - \frac{\delta^2 \bar A_\eta^2}{2(1+\delta^2)}  \right] 
\end{equation}
where $d$ is the dimensionless constant
$$ d=  \frac{(2\pi)^{3/2}  \mu_0 \, a}{ \zeta} \left(\sum_\pm   \,  \frac{ |q^\pm|}{ m^\pm} \right) >0$$
Alternatively, the left side of \eqref{E:sp09} can be expanded as follows
\begin{equation}\label{E:sp10}
 -\partial_{\bar r}^2\bar A_\eta - \frac{1}{\bar r}\partial_{\bar r}\bar  A_\eta  + ( \frac{3\bar r^2}{\delta^4} - \frac{2}{\delta^2})\bar A_\eta =  - d \, \frac{  \delta^2 \bar A_\eta}{ (1+\delta^2)^{3/2}} \exp\left[ - \frac{\delta^2 \bar A_\eta^2}{2(1+\delta^2)} \right]
\end{equation}
We obtain Theorem \ref{T:screw-pinch-var} by taking $\gamma^\pm =\pm 1$.

\section{Inclusion of an external field\label{S:external}}

Given that there are no finite energy equilibria (see Appendix \ref{Appendix A}), we now consider the inclusion of an external field in the VM system \eqref{E:Vlasov}, \eqref{E:Maxwell}, \eqref{E:charge_current}.  The Vlasov equation \eqref{E:Vlasov} is modified to 
\begin{equation}
\label{E:VlasovExt}
\partial_t f^{\pm} + \bds v\cdot \nabla_{\bds x} f^\pm + \frac{q^\pm}{m^\pm}(\bds E +\bds E_\ext + \bds v \times (\bds B+\bds B_\ext)) \cdot \nabla_{\bds v} f^\pm = 0
\end{equation}
but the Maxwell equations are expressed entirely in terms of $\bds E$ and $\bds B$ (with no $\bds E_\ext$ or $\bds B_\ext$) exactly as in  \eqref{E:Maxwell}, and the charge and current densities are given by \eqref{E:charge_current}.  As in \eqref{E:potentials}, potentials $(\phi, \bds A)$ are introduced corresponding to $(\bds E, \bds B)$ and external potentials $(\phi_\ext,\bds A_\ext)$ corresponding to $(\bds E_\ext, \bds B_\ext)$.   

The implications for equilibria are the following.  The particle trajectories now include the external field components:
\begin{align*}
&e^\pm = \tfrac12 m^\pm (v_r^2 + v_\theta^2 + v_z^2) + q^\pm (\phi+\phi_\ext)\,, \\
&p_z^\pm = m^\pm v_z + q^\pm(A_z + A_{z,\ext}) \,, \\
&\ell_z^\pm = r(m^\pm v_\theta + q^\pm(A_\theta + A_{\theta,\ext}))
\end{align*}
Thus the equations for the equilibria in the $z$-pinch case are only altered in the following manner: $\rho(\phi,A_z)$ is replaced by $\rho(\phi+\phi_\ext, A_z + A_{z,\ext})$ and $j_z(\phi,A_z)$ is replaced by $j_z(\phi+\phi_\ext, A_z + A_{z,\ext})$.  We take $\phi_\ext=0$ and then arrange for $\phi=0$ and $\rho\equiv 0$ in the same way as above.  The current equation is only altered on the right-side -- the nonlinear term is evaluated at $A_z+A_{z,\ext}$ in place of $A_z$ alone.  The result, after the same rescaling, is 
\begin{equation}
\label{E:zp09f}
-\partial_{\bar r}^2 \, \bar A_z - \frac1{\bar r} \, \partial_{\bar r} \, \bar A_z = - d \,(\bar A_z+\bar A_{z,\ext}) \,e^{-\frac12(\bar A_z+\bar A_{z,\ext})^2}
\end{equation}
in place of \eqref{E:zp09}.  

We explain that, given a nontrivial choice of $\bar A_{z,\ext}$ (say, smooth with compact support), we expect that there exist $\bar A_z$ such that $(\bar A_z, \bar A_{z,\ext})$ solves \eqref{E:zp09f}.  To see this, let us drop the bar and $z$ subscript and set $d=1$ so that \eqref{E:zp09f} becomes
\begin{equation}
\label{E:zp10}
-\partial_{ r}^2 \,  A - \frac1{ r} \, \partial_{ r} \,  A = - ( A+ A_{\ext}) \,e^{-\frac12( A+ A_{\ext})^2}
\end{equation}
Now let $A_\tot = A+A_\ext$ so that \eqref{E:zp10} becomes
\begin{equation}
\label{E:zp11}
-\Delta A_\tot + A_\tot e^{-A_\tot^2/2} = - \Delta A_\ext
\end{equation}
where $\Delta$ is the 2D Laplacian. 
Taking the right side as a given forcing function, 
\eqref{E:zp11} becomes
\begin{equation}\label{E:zp12}
(I-\Delta) A_\tot  + \underbrace{A_\tot(e^{-A_\tot^2/2}-1)}_{\text{cubic+ term}} = -\Delta A_\ext 
\end{equation}
The kernel corresponding to $(I-\Delta)^{-1}$ has exponential decay and we can define a first approximation to $A_\tot$ as
$$A_\tot = (I-\Delta)^{-1} (-\Delta A_\ext)$$
To obtain an exact solution to \eqref{E:zp12}, we can proceed analytically assuming $f$ is small or numerically via an iteration procedure.  An analytical tool could be Lyapunov-Schmidt reduction or Crandall-Rabinowitz theory.  Numerically,
\begin{enumerate}[left=0pt,label=$\bullet$]
\item Petviashvili iteration 
\item a scheme in which the resolvent kernel is iteratively updated to include an effective potential term given by the previous iterate plugged into $e^{-A_\tot^2/2}-1$.
\end{enumerate}

\section{Adaptations for the relativistic case\label{S:relativistic}}

We provide here some comments to indicate that our results carry over to the relativistic case, although we have not computed a complete set of constructive equations.

First, we discuss particle trajectories in the relativistic setting.  Let $c$ be the speed of light, and consider a single particle in the field with charge $q$ and mass $m$.  The Hamiltonian is
$$H(\bds x, \bds p) = mc^2 (\gamma-1) + q\phi(\bds x) \,, \qquad \gamma = \sqrt{ 1 + \frac{ |\bds p-q\bds A(\bds x)|^2}{m^2c^2}}$$
Let $\bds v$ be the \emph{kinetic momentum divided by mass}, and $\hat{\bds v}$ be the particle velocity:
\begin{equation}
\label{E:rel-v}
\bds v = \frac{\bds p - q\bds A(\bds x)}{m} \,, \qquad \hat{\bds v} = \frac{\bds v}{\gamma}
\end{equation}
In this language,
$$\gamma = \sqrt{1+ \frac{|\bds v|^2}{c^2}} \quad \implies \quad \hat{\bds v} = \frac{\bds v}{\sqrt{1 + \frac{|\bds v|^2}{c^2}}}$$
from which we see that $|\hat{\bds v}|<c$.  Although in physics $\bds v$ usually denotes the velocity, our convention of using $\bds v$ for the kinetic momentum divided by mass gives \eqref{E:rel-v} which is the same equation relating the symbols $\bds v$ and $\bds p$ as in the nonrelativistic case.  

The equations of motion become
$$
\begin{aligned}
&\dot{\bds x} = \hat{\bds v} \\
& m \dot{\bds v} = q(\bds E + \hat{\bds v} \times \bds B)
\end{aligned}
$$

In the context of $z$-pinch, $\theta$-pinch, or screw-pinch equilibria with $\phi=0$, we can create the same quadratic structure in velocities inside the exponential expression for the densities as in the non-relativistic since now
$$\Big( \frac{e}{mc^2} + 1\Big)^2 - 1 = \frac{|\bds v|^2}{c^2}$$
while $p_z$, $\ell_z$ are given by the same formula as in the non-relativistic case.  Thus, to convert from nonrelativistic to relativistic formulae, in each defining expression for $\mu^\pm$ in the paper, replace the the particle energy $e$ by $\frac12mc^2[(e/mc^2+1)^2-1]$.

Also, in the relativistic system the charge and current densities are
\[
 \rho=\sum_\pm q^\pm\int f^\pm(\bds{x},\bds{v})\,dv,
 \qquad
 \bds{j}=\sum_\pm q^\pm\int \hat{\bds{v}} f^\pm(\bds{x},\bds{v})\,dv ,
\]
where $\hat{\bds{v}}$ is the relativistic velocity.  Thus the nonlinear relations determining $A_z$, $A_\theta$, and $A_\eta$ are modified by relativistic velocity moments.  Although the invariant structure of the equilibria remains the same, the reduced nonlinear ordinary differential equations must be rederived from these relativistic moments.

\appendix

\section{Absence of finite-energy Vlasov-Maxwell equilibria on $\R^2 \times \T$ without external fields}
\label{Appendix A}




In this section, we show that there is no finite-energy equilibria for the Vlasov-Maxwell equation on $\R^2 \times \T$ without external fields.

Let
\[
 \Om=\R^2\times\T,
 \qquad
 x=(x_1,x_2,x_3)=(\xperp,z).
\]
We allow finitely many particle species, indexed by $\alpha=1,\dots,N$, with masses $m_\alpha>0$, charges $q_\alpha\in\R$, and nonnegative distribution functions $f^\alpha$.  Set
We consider a time-independent Vlasov--Maxwell state without external fields (we omit some physical constants for simplicity)
\begin{align}
 \bds{v}\cdot\nabla_{\bds{x}} f^\alpha
 +\frac{q_\alpha}{m_\alpha}\bigl(\bds{E}+ \bds{v} \times \bds{B}\bigr)
 \cdot\nabla_{\bds{v}} f^\alpha&=0,
 \label{eq:stationary-vlasov}\\
 \curl \bds{B}=\bds{j},\qquad \curl \bds{E}&=0,
 \label{eq:stationary-maxwell-curl}\\
 \diver \bds{E}=\rho,\qquad \diver \bds{B}&=0,
 \label{eq:stationary-maxwell-div}
\end{align}
where
\begin{equation}
 \rho=\sum_{\alpha=1}^N q_\alpha\int_{\R^3}f^\alpha\,\dd \bds{v},
 \qquad
 \bds{j}=\sum_{\alpha=1}^N q_\alpha\int_{\R^3}  \bds{v} f^\alpha\,\dd \bds{v}.
 \label{eq:rho-j}
\end{equation}
The energy per period is
\begin{equation}
 \mathcal E
 =\int_{\Om}\left\{
 \frac12\bigl(|\bds{E}|^2+|\bds{B}|^2\bigr)
 +\sum_{\alpha=1}^N\int_{\R^3} (m_\alpha^2+|\bds{v}|^2) f^\alpha(\bds{x},\bds{v})\,\dd \bds{v}
 \right\}\dd \bds{x}.
 \label{eq:energy}
\end{equation}


We have
\begin{proposition}[No finite-energy equilibrium on $\R^2\times\T$]
\label{thm:main}
Let $(f^1,\dots,f^N,\bds{E},\bds{B})$ be a time-independent solution of
\eqref{eq:stationary-vlasov}--\eqref{eq:stationary-maxwell-div} on
$\Om\times\R^3$.  Assume that
\begin{enumerate}
 \item $f^\alpha\ge 0$ is a phase-space density with respect to Lebesgue measure, and the integrations by parts used below are valid; this holds, for example, for a classical equilibrium with sufficient decay in $v$;
 \item $\bds{E}$, $\bds{B}$ are periodic in $z$ and the energy \eqref{eq:energy} is finite.
\end{enumerate}
Then
\[
 f^\alpha=0\quad(\alpha=1,\dots,N),
 \qquad \bds{E}=0,\qquad \bds{B}=0
\]
almost everywhere.  In particular, there is no nontrivial finite-energy VM equilibrium on $\R^2\times\T$ in the absence of an external field.
\end{proposition}

The proof uses only the stationary momentum balance and therefore also applies to weak equilibria whenever that balance is available distributionally and the stress tensor is integrable.


We first record the standard local momentum identity.  The particle stress tensor is
\begin{equation}
 \Pi_{ij}
 =\sum_{\alpha=1}^N
 m_\alpha \int_{\R^3} v_i v_j f^\alpha\,\dd \bds{v}.
 \label{eq:particle-stress}
\end{equation}
The Maxwell stress tensor in the sign convention convenient for momentum balance is
\begin{equation}
 \Sigma_{ij}
 =E_iE_j+B_iB_j
 -\frac12\bigl(|\bds{E}|^2+|\bds{B}|^2\bigr)\delta_{ij}.
 \label{eq:maxwell-stress}
\end{equation}
Thus the total flux of momentum is
\begin{equation}
 T_{ij}=\Pi_{ij}-\Sigma_{ij}
 =\Pi_{ij}
 +\frac12\bigl(|\bds{E}|^2+|\bds{B}|^2\bigr)\delta_{ij}
 -E_iE_j-B_iB_j.
 \label{eq:total-stress}
\end{equation}
Note that the energy assumption makes $T_{ij}$ integrable.  

\begin{lemma}[Stationary stress balance]
\label{lem:stress-balance}
Under the hypotheses of Proposition \ref{thm:main},
\begin{equation}
 \partial_{x_j}T_{ij}=0
 \qquad (i=1,2,3)
 \label{eq:stress-div-free}
\end{equation}
in the sense of distributions on $\Om$.
\end{lemma}

\begin{proof}
Multiply \eqref{eq:stationary-vlasov} by $v_i$ and integrate in $\bds{v}$.  Since
\[
 \nabla_{\bds{v}}\cdot\bigl(\bds{E}+ \bds{v} \times \bds{B}\bigr)=0,
\]
integration by parts gives
\begin{equation}
 \partial_{x_j}\Pi_{ij}
 =\rho E_i+(\bds{j}\times \bds{B})_i.
 \label{eq:particle-force-balance}
\end{equation}

On the other hand, the stationary Maxwell equations imply the familiar identity
\begin{equation}
 \partial_{x_j}\Sigma_{ij}
 =\rho E_i+(\bds{j}\times \bds{B})_i.
 \label{eq:maxwell-force-balance}
\end{equation}
Indeed, this follows from
\[
 (F\cdot\nabla)F-\frac12\nabla|F|^2
 =-(F\times\curl F)
\]
together with $\diver \bds{E}=\rho$, $\curl \bds{E}=0$, $\diver \bds{B}=0$, and
$\curl \bds{B}=\bds{j}$.  Subtracting \eqref{eq:maxwell-force-balance} from
\eqref{eq:particle-force-balance} yields \eqref{eq:stress-div-free}.
\end{proof}




The central algebraic observation is that the trace over the two noncompact directions is nonnegative.

\begin{lemma}[Positive transverse trace]
\label{lem:transverse-trace}
The stress tensor \eqref{eq:total-stress} satisfies
\begin{equation}
 T_{11}+T_{22}
 =m_\alpha\sum_{\alpha=1}^N\int_{\R^3}
 (v_1^2+v_2^2)  f^\alpha\,\dd \bds{v}
 +|E_3|^2+|B_3|^2 . 
 \label{eq:positive-trace}
\end{equation}
In particular, $T_{11}+T_{22}\ge0$.
\end{lemma}

\begin{proof}
Summing the first two diagonal entries in \eqref{eq:total-stress} gives
\begin{align*}
 T_{11}+T_{22}
 &=m_\alpha\sum_{\alpha=1}^N\int_{\R^3}
 (v_1^2+v_2^2) f^\alpha\,\dd \bds{v} \\
 &\quad+\bigl(|\bds{E}|^2+|\bds{B}|^2\bigr)
 -\bigl(E_1^2+E_2^2+B_1^2+B_2^2\bigr),
\end{align*}
which is exactly \eqref{eq:positive-trace}.
\end{proof}

\begin{proposition}[Rigorous stationary transverse virial identity]
\label{prop:virial}
Under the hypotheses of Proposition \ref{thm:main},
\begin{equation}
 \int_{\Om}\left\{
 \sum_{\alpha=1}^N m_\alpha\int_{\R^3}
  (v_1^2+v_2^2) f^\alpha\,\dd \bds{v}
 +|E_3|^2+|B_3|^2
 \right\}\dd \bds{x}=0.
 \label{eq:stationary-virial}
\end{equation}
\end{proposition}

\begin{proof}
Choose $\eta\in C_c^\infty([0,\infty))$ with
$0\le\eta\le1$, $\eta(s)=1$ for $0\le s\le1$, and $\eta(s)=0$ for
$s\ge2$.  For $R>1$, set $r=\sqrt{x_1^2+x_2^2}$ and define the periodic vector field
\begin{equation}
 X_R(x)=\eta(r/R)(x_1,x_2,0).
 \label{eq:cutoff-vector-field}
\end{equation}
It is compactly supported in the noncompact variables, and
\[
 \sup_{R>1}\|\nabla X_R\|_{L^\infty}<\infty.
\]
Testing \eqref{eq:stress-div-free} against $X_R$ gives
\begin{equation}
 0=-\int_{\Om}T_{ij}\partial_{x_j}(X_R)_i\,\dd \bds{x}.
 \label{eq:cutoff-stress}
\end{equation}
For every fixed $x$, as $R\to\infty$,
\[
 \partial_{x_j}(X_R)_i
 \longrightarrow
 \begin{cases}
  \delta_{ij},&i,j\in\{1,2\},\\
  0,&\text{otherwise}.
 \end{cases}
\]
Note that the energy assumption makes every component of $T$ integrable, so $T\in L^1(\Om)$. Dominated convergence in
\eqref{eq:cutoff-stress} yields
\[
 \int_{\Om}(T_{11}+T_{22})\,\dd \bds{x}=0.
\]
Now apply Lemma \ref{lem:transverse-trace}.
\end{proof}



We first use \eqref{eq:stationary-virial} to eliminate the particles.
Because every term in its integrand is nonnegative, for each species $\alpha$,
\begin{equation}
 \int_{\Om}\int_{\R^3}
 (v_1^2+v_2^2) f^\alpha(\bds{x},\bds{v})\,\dd \bds{v}\dd \bds{x}=0.
 \label{eq:particle-zero-integral}
\end{equation}
Hence $f^\alpha(\bds{x},\bds{v})=0$ for almost every $(\bds{x},\bds{v})$ with $(v_1, v_2 \ne (0, 0)$.
The set
\[
 \{v\in\R^3:v_1=v_2=0\}
\]
has three-dimensional Lebesgue measure zero.  Since $f^\alpha$ is a phase-space density, it follows that
\begin{equation}
 f^\alpha=0\quad\text{almost everywhere}
 \qquad(\alpha=1,\dots,N).
 \label{eq:f-zero}
\end{equation}
In particular, $\rho=0$, $\bds{j}=0$.  The same virial identity also gives
$E_3=B_3=0$, although this will not be needed in the final step.

It remains to eliminate finite-energy vacuum Maxwell fields.  We record the elementary $L^2$ Hodge lemma appropriate to the cylinder.  The proof is standard so we omit it.

\begin{lemma}[$L^2$ harmonic vector fields on the cylinder]
\label{lem:hodge}
If $F\in L^2(\R^2\times\T;\R^3)$ satisfies
\[
 \curl F=0,
 \qquad
 \diver F=0
\]
in the sense of distributions, then $F=0$ almost everywhere.
\end{lemma}


After \eqref{eq:f-zero}, equations
\eqref{eq:stationary-maxwell-curl}--\eqref{eq:stationary-maxwell-div} reduce to
\[
 \curl \bds{E}=\diver \bds{E}=0,
 \qquad
 \curl \bds{B}=\diver \bds{B}=0.
\]
Since finite energy gives $\bds{E}$, $\bds{B}\in L^2(\Om)$, Lemma \ref{lem:hodge} implies
$\bds{E}=\bds{B}=0$.  This proves Proposition \ref{thm:main}.

\begin{remark}[Relativistic analogue]
The same argument applies to the relativistic Vlasov--Maxwell system.  Denote 
\[
p^0_\alpha(\bds{v})=\sqrt{m_\alpha^2+|\bds{v}|^2} , \qquad \hat{\bds{v}}_\alpha=\frac{\bds{v}}{p^0_\alpha(\bds{v})} .
\]
One replaces the particle stress by
\[
\Pi_{ij}
 =\sum_{\alpha=1}^N
 \int_{\R^3}v_i\widehat v_{\alpha,j}f^\alpha\,\dd \bds{v}
 =\sum_{\alpha=1}^N
 \int_{\R^3}\frac{v_iv_j}{p^0_\alpha(\bds{v})}f^\alpha\,\dd \bds{v} , 
\]
and its transverse trace is
\[
 T_{11}+T_{22}
 =\sum_{\alpha=1}^N\int_{\R^3}
 \frac{|\vperp|^2}{p^0_\alpha(\bds{v})}f^\alpha\,\dd \bds{v}
 +|E_3|^2+|B_3|^2.
 \label{eq:positive-trace}
\]
The rest of the proof is unchanged.
\end{remark}


\bibliographystyle{amsalpha}
\bibliography{references}

\end{document}